\documentclass[conference, onecolumn]{IEEEtran}
\makeatletter
\def\ps@headings{%
\def\@oddhead{\mbox{}\scriptsize\rightmark \hfil \thepage}%
\def\@evenhead{\scriptsize\thepage \hfil \leftmark\mbox{}}%
\def\@oddfoot{}%
\def\@evenfoot{}}
\makeatother 

\usepackage{verbatim}   
\usepackage{algorithm}
\usepackage{algorithmic}
\usepackage{times}
\usepackage{amsmath}\usepackage{amssymb}
\usepackage{latexsym}
\usepackage{graphicx, color}
\usepackage{epsfig}
\usepackage{url}
\usepackage{caption}
\usepackage{psfrag}
\usepackage{cite}
\usepackage{subfigure}

\renewcommand{\baselinestretch}{1}

 \newcommand{\bnum}{\begin{enumerate}}
 \newcommand{\enum}{\end{enumerate}}
 \newcommand{\ba}{\begin{array}}
 \newcommand{\ea}{\end{array}}
 \newcommand{\bitm}{\begin{itemize}}
 \newcommand{\eitm}{\end{itemize}}
 \newcommand{\beq}{\begin{equation}}
 \newcommand{\eeq}{\end{equation}}
 \newcommand{\beqn}{\begin{eqnarray}}
 \newcommand{\eeqn}{\end{eqnarray}}
 \newcommand{\beqno}{\begin{eqnarray*}}
 \newcommand{\eeqno}{\end{eqnarray*}}
 \newcommand{\bma}{\begin{displaymath}}
 \newcommand{\ema}{\end{displaymath}}
 \newcommand{\bnu}{\begin{enumerate}}
 \newcommand{\enu}{\end{enumerate}}
 \newcommand{\bce}{\begin{center}}
 \newcommand{\ece}{\end{center}}
 \newcommand{\btb}{\begin{tabular}}
 \newcommand{\etb}{\end{tabular}}
 \newcommand{\bmat}{\begin{pmatrix}}
 \newcommand{\emat}{\end{pmatrix}}

\newtheorem{theorem}{\textbf{Theorem}}
\newtheorem{lemma}{\textbf{Lemma}}

\newtheorem{remark}{\textbf{Remark}}
\newtheorem{definition}{Definition}

\begin{document}

\title{\LARGE \bf Scheduling in Parallel Queues with Randomly Varying Connectivity and Switchover Delay}

\author{Guner D. Celik, Long B. Le and Eytan Modiano}

\maketitle

\begin{abstract}
We consider a dynamic server control problem for two parallel queues
with \emph{randomly varying connectivity} and \emph{server
switchover time between the queues}. At each time slot the server
decides either to stay with the current queue 
or switch to the other queue based on the current connectivity and
the queue length information. 
The introduction of switchover time is a new modeling component of
this problem, which makes the problem much more challenging. We
develop a novel approach to characterize the stability region of the
system by using \emph{state action frequencies}, which are
stationary solutions to a Markov Decision Process (MDP) formulation
of the corresponding saturated system. We 
characterize the stability region explicitly in terms of the
connectivity parameters and develop a 
frame-based
dynamic control (FBDC) policy that is shown to be
throughput-optimal. In fact, the FBDC policy 
provides \emph{a new framework for developing
throughput-optimal network control policies} using state action
frequencies. 
Further, we develop simple \emph{Myopic policies} that achieve more
than $96\%$ of the stability region. 
Finally, simulation results show that the Myopic policies may
achieve the full stability region and are more delay efficient than
the FBDC policy in most cases.
\end{abstract}

\section{Introduction}

Scheduling a dynamic server over randomly varying wireless channels
has been a very popular topic since the seminal works by Tassiulas
and Ephremides in \cite{tass92} and \cite{tass93}. These works were
generalized to many different settings by several authors in the
network control field (e.g.,
\cite{EryilOzMod07,LinShr05,ModShahZuss06,neely03,neely05,shakk08,Stolyar04,WuSri06}).
However, the significant effect of server switchover time between
the queues has been ignored. We consider a parallel queue network
with \emph{randomly varying connectivity} and \emph{the server
switchover time} between the queues and study the impact of the
switchover time on the system performance.

Our model consists of two parallel queues whose connectivity
is varying in time according to a stochastic process
and one server receiving data packets from the queues by dynamically
adjusting its
position as shown in Fig.~\ref{Fig:two_queues_switching}. 
We consider a slotted system where the slot length is equal to a
packet transmission time and it takes one slot for
the server to switch from one queue to the other. 
A packet is successfully received from queue-$i$ if queue-$i$ is
connected, if the server is present at queue-$i$ and if it decides
to stay at queue-$i$. Therefore, the server is to dynamically choose
to stay with the current queue or switch to the other queue based on
the connectivity and the queue length information of both queues. To
the best of our knowledge, this paper is the first to consider
random connectivity and switchover times to be simultaneously
present in the system. Our purpose is to characterize the effect of
switchover time on system performance. In particular, we are
interested in the impact of the switchover time on the maximum
throughput region (or the throughput region for simplicity) and to
find the optimal scheduling policy for the server that stabilizes
the system whenever the arrivals are within the throughput region.

Switchover delay in dynamic server control problems is a
widespread phenomenon that can be observed in many practical
systems. In satellite systems where a mechanically steered antenna
is providing service to 
ground stations, the time to
switch from one station to another can be around 10ms \cite{BlakeLong09}, \cite{TolShu96}.
Similarly, the delay for electronic beamforming can
be on the order of $10\mu s$ in wireless radio systems
\cite{BlakeLong09}, \cite{TolShu96}.
Furthermore, in optical
communication systems tuning delay for transceivers can take
significant time ($\mu$s-ms) \cite{Berz}, \cite{ModBarry00}. We show in this paper
that switchover delay indeed fundamentally changes the system
characteristics. As compared to the seminal work of Tassiulas and
Ephremides in \cite{tass93},  the supported rate region shrinks
considerably, the optimal policies change and novel mathematical
approaches might be necessary for systems with nonzero switchover
delay.

\begin{figure}
\centering \psfrag{4\r}[l][][1]{$\!\!\lambda_1$}
\psfrag{5\r}[l][][1]{$\lambda_2$}
\psfrag{1\r}[l][][0.85]{\!\!\!\!\!\!\!\!Server}
\psfrag{2\r}[l][][1]{$\!\!\!\!\!\! C_1$}
\psfrag{3\r}[l][][1]{$\!\!\!\! C_2$} \psfrag{6\r}[l][][1]{$\!\!t_s$}
\includegraphics[width=0.2\textwidth]{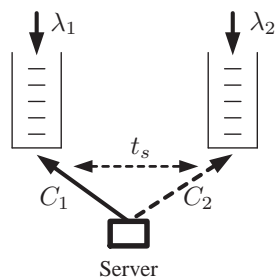}
\caption{System model. Two queues with randomly varying
connectivities ($C_1$ and $C_2$) and $t_s=1\; \textrm{slot}$
switchover time. The server is currently connected to queue-1 and it
takes 1 time slot to switch to
queue-2.}\label{Fig:two_queues_switching}
\end{figure}

Note that 
the switchover time can be smaller or larger than a packet
transmission duration in practical systems. In systems where it is
less than 1 slot, when the server switches from one queue to
another, it usually has to waste the entire slot due to
synchronization issues. For systems with significant switching times
(e.g., vehicular networks with mobile relays), our analysis can be
used as a starting point while keeping in mind that similar solution
techniques will apply. Finally note that some of our results, in
particular the FBDC policy, and the throughput region
characterization in terms of state action frequencies hold for more
general systems such as many queues with arbitrary switchover times
and channel statistics.

We analytically characterize the throughput region
$\mathbf{\Lambda}$: The set of all arrival rate pairs
($\lambda_1,\lambda_2$) that the system can stably support. We
derive necessary and sufficient stability conditions on the arrival
rate pairs $(\lambda_1,\lambda_2)$ in terms of the connectivity
parameters for both
correlated and uncorrelated connectivity processes. 
For this, we consider the corresponding saturated system in which
there is always a packet to send in both queues and we formulate a
discrete time Markov Decision Process (MDP) whose stationary
deterministic solutions in terms of \emph{state action
frequencies} provide corner points of the
polytope of achievable rates, i.e., the throughput region. We
develop a 
frame based dynamic control (FBDC) policy
for the original system with dynamic arrivals. FBDC policy is based
on solving a Linear Programm (LP) corresponding to the MDP solution
for the saturated system and it is throughput-optimal asymptotically
in the frame length. FBDC policy is applicable to many general systems
and provides \emph{a new framework} for developing
throughput-optimal policies for network control. Namely, for any
system whose corresponding saturated system is Markovian with finite
state space, FBDC policy achieves stability by solving an LP to find
the stationary MDP solution of the saturated system and applying
this solution over a frame in the actual system. We also develop
simple Myopic policies with throughput guarantees that do not
require the solution of an LP and that can be more delay efficient than the FBDC policy.
We show that the Myopic policy with ``one lookahead'' 
achieves at least $90\%$ of the throughput region while the Myopic policies
with
2 and 3-lookahead achieve more than $94\%$ and $96\%$ of the
stability region respectively.
%
%
The mathematical solution technique used for proving the stability
of various policies is novel in this paper in that it involves
utilizing Markov Decision Theory inside the Lyapunov stability
arguments.

%

Optimal control of queueing systems and communication networks has
been a very active research topic over the past two decades. In the
the seminal paper \cite{tass92}, Tassiulas and Ephremides
characterize the stability region and propose the well-known
max-weight scheduling algorithm. 
Later in \cite{tass93}, they consider a parallel queueing system
with randomly varying connectivity and prove the
throughput-optimality of the Longest-Connected-Queue scheduling
policy. These results are extended to the joint power allocation and
routing problem in wireless networks in \cite{neely03} and
\cite{neely05} and the optimal scheduling problem for 
switches in \cite{ShahWis06} and \cite{Stolyar04}. Decentralized and
greedy scheduling algorithms with throughput guarantees are studied
in \cite{ChapKarSar05}, \cite{Long10}, \cite{LinShr05},
\cite{WuSri06}, while \cite{EryilOzMod07} and \cite{ModShahZuss06}
consider distributed algorithms that achieve throughput-optimality
(see \cite{Geor_Neely_Tass06} for a detailed review).
In \cite{KarLuo07}, \cite{AnnaEphr09} and \cite{shakk08} the network
control problem with delayed channel state information is studied,
while \cite{AhmadTaraKrish09} and \cite{LiNeely10} investigate
network control with limited channel sensing. These existing works
do not consider the server switchover times. Scheduling in optical
networks under reconfiguration latency was considered in
\cite{Berz}, where the transmitters and receivers were assumed to be
unavailable during the system reconfiguration time. While switchover
delay has been studied in polling models in the queueing theory
community (e.g., \cite{AltKonsLiu92}, \cite{Levy},
\cite{LiuNainTow}, \cite{vish06}), random connectivity was not
considered since it may not arise in classical polling applications.
To the best of our knowledge, this paper is the first to
simultaneously consider random connectivity and server switchover
times.


The main contribution of this report is solving the scheduling
problem in parallel queues with \emph{randomly varying connectivity}
and \emph{server switchover times} for the first time. In
particular,
\begin{itemize}
\item
We establish the stability region of the 
system
using the \emph{state action frequencies} 
of the MDP 
formulation for the corresponding saturated system. Furthermore, we
characterize the stability region explicitly in terms of the
connectivity parameters.

\item
We develop a 
frame-based dynamic control (FBDC)
policy and show that 
it is throughput-optimal asymptotically in the frame length. The
FBDC policy is applicable to more general systems whose
corresponding saturated system is Markovian with finite state and
action spaces, for example, networks with more than two queues,
arbitrary switchover times and general 
arrival 
and  Markov modulated channel processes.

\item
We develop a simple 1-Lookahead Myopic policy that achieves 
at least $90\%$ of the stability region while the Myopic policies
with
2 and 3-lookahead achieve more than $94\%$ and $96\%$ of the
stability region respectively.

\item
We present 
simulations suggesting that the Myopic policies may be
throughput-optimal and are more delay efficient than the
throughput-optimal FBDC policy in most cases.
\end{itemize}
This paper provides \emph{a novel framework for solving network
control problems} via
characterizing the stability region 
in terms of state action frequencies 
and achieving throughput-optimality by utilizing the state action
frequencies over frames.

In the next section we introduce the system model and in Section
\ref{Sec:Motiv} we provide a motivating example by analyzing the
case with uncorrelated channel processes over time. We establish the
throughput region in Section \ref{Sec:Stab_Reg} via formulating a
MDP for the saturated system. We prove the throughput optimality of the
FBDC policy in Section
\ref{Sec:Opt_Pol} and 
analyze simple Myopic policies with large throughput guarantees in
Section \ref{Sec:Myopic_Pol}. We provide simulation results in Section \ref{Sec:Sim} and conclude
in Section \ref{Sec:Conc}.



\section{The Model}\label{Sec:Model}

Consider two parallel queues with randomly varying connectivity and
one server receiving data packets from the queues. 
Time is slotted into unit-length time slots 
equal to one
packet transmission time; $t \in \{0,1,2,...\}$. 
It takes one slot for the server to switch from one queue to the
other, and $m(t)$ denotes
the queue at which the server is present at 
slot $t$. 
Let the stationary stochastic process  $A_i(t)$, with average
arrival rate $\lambda_i$, denote the number of packets arriving to
queue $i$ at time slot $t$ where
$\mathbb{E}[A_i^2(t)]\le
A_{\max}^2$, $i \in \{1,2\}$. 
Let $\mathbf{C}(t)=(C_1(t),C_2(t))$
be the channel (connectivity) process 
at time slot $t$, where $C_i(t)= 0$ for the OFF state (disconnected)
and $C_i(t)= 1$ for the ON state (connected).  We assume that the
processes $A_1(t),A_2(t),C_1(t)$ and $C_2(t)$ are independent. 

We analyze
two different models for the connectivity process $\mathbf{C}(t)$:
\begin{definition}[Uncorrelated Channels \cite{neely03}, \cite{neely05inf}, \cite{tass93}]\label{def:uncorr_ch}
\emph{The process $C_i(t)$, $i \in \{1,2\}$, is in ON state with
probability (w.p.) $p_i$ and in OFF state w.p. $1-p_i$ at each time
slot independently from earlier slots and of the other queue.}
\end{definition}
\begin{definition}[Correlated Channels \cite{AhmadTaraKrish09}, \cite{LiNeely10},
\cite{WangChang96}, \cite{ZorziRao95}]\label{def:corr_ch} \emph{The
process $C_i(t)$, $i \in \{1,2\}$, follows the two-state Markov
chain (i.e., the symmetric Gilbert-Elliot channel model) with transition
probability $\epsilon$ as shown in Fig. \ref{Fig:channel_MC}
independently of the other queue.}
\end{definition}
G-E channel model 
has been widely accepted in modeling and
analysis
of wireless systems \cite{AhmadTaraKrish09}, \cite{LiNeely10}, 
\cite{WangChang96}, \cite{ZorziRao95}, \cite{ZorziRao97}.
Note that our results and algorithms are applicable to general
non-symmetric channel models, but here we present the symmetric case
for ease of exposition.

Let $\mathbf{Q}(t)=(Q_1(t),Q_2(t))$ be the queue lengths at time
slot $t$. We assume that $\mathbf{Q}(t)$ and $\mathbf{C}(t)$ are
known to the
server at the beginning of each time slot.  
Let $a_t \in \{0,1\}$ denote the action taken at slot $t$, where 
$a_t=1$ if the server stays with the current queue and $a_t=0$ if it
switches to the other queue. One packet is successfully received
from queue $i$ at time slot $t$, if $m(t)=i$,
$a_t=1$ and $C_i(t)=1$. 

%
%
%
%
%
%
%
%
\begin{definition}[Strong Stability]\label{def:stab}
\emph{A queue is called strongly stable if :}
\begin{equation*}
\displaystyle \limsup_{t\rightarrow \infty} \frac{1}{t}
\sum_{\tau=0}^{t-1} \mathbb{E}[Q(\tau)] < \infty.
\end{equation*}
\end{definition}
In addition, the system is called strongly stable (or stable for
simplicity) if both queues are stable. 
%
\begin{definition}[Stability Region]\label{def:throughput_region}
\emph{The stability region $\mathbf{\Lambda}$ is the set of all
arrival rate vectors $(\lambda_1,\lambda_2)$ such that there exists
a control algorithm that stabilizes both queues in the system.}
\end{definition}
%
The $\mathbf{\delta}$-stripped stability region is defined for some
$\delta >0$ as $\mathbf{\Lambda}^{\delta} \triangleq \Big\{
(\lambda_1,\lambda_2)| (\lambda_1+\delta, \lambda_2+\delta) \in
\mathbf{\Lambda} \Big\}.$ A policy is said to achieve
$\gamma$-fraction of $\mathbf{\Lambda}$, if it stabilizes the system
for all input rates inside $\gamma \mathbf{\Lambda}$. A
throughput-optimal policy achieves
$\gamma=100\%$ of the stability region. 

%
%
%

%
%

\begin{figure}
\centering \psfrag{2\r}[l][][1]{$\!\!\epsilon$}
\psfrag{6\r}[l][][1]{$\epsilon$} \psfrag{4\r}[l][][0.85]{\!\!\!\!ON}
\psfrag{5\r}[l][][0.85]{\!\!\!\!\!OFF}
\psfrag{1\r}[l][][1]{\!\!\!\!\!$\!\!1-\epsilon$}
\psfrag{3\r}[l][][1]{\!\!\!\!\!$\!\!\!\!1-\epsilon$}
\includegraphics[width=0.23\textwidth]{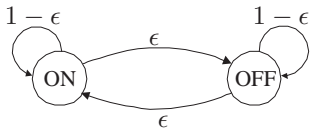}
\vspace{-2mm} \captionsetup{font=footnotesize}\caption{Markov
modulated ON/OFF channel process. We have $\epsilon \le 0.5$ for
positive correlation.}\label{Fig:channel_MC}
\end{figure}

\section{Motivation-Uncorrelated Channels}\label{Sec:Motiv}
In this section we show that there is no diversity gain when the
channel processes are i.i.d. over time and that channel correlation
over time is necessary in order to take advantage of the diversity
gain and enlarge the throughput region. Specifically, we show that
when the channel processes are i.i.d. over time, the stability
region is reduced considerably with respect to the no-switchover
time case, and no policy can achieve a stability region larger than
that of the simple Exhaustive or Gated type policies. Gated policy
is such that the server serves all the packets that were present at
the queue at the time of arrival and then switches to the other
queue. In the Exhaustive policy, the server does not leave the
current queue until it empties.

Assume the channel processes $C_1(t)$ and $C_2(t)$ are as described
in Definition \ref{def:uncorr_ch}. We first derive a necessary
condition on the stability of the system and then show the
sufficiency of this condition by proving that gated policy
stabilizes the system under this condition.
\begin{theorem}\label{thm:iid_ness}
\emph{A necessary condition on stability is given by:}
\begin{equation}\label{eq:Stab_Necess}
\displaystyle \rho = \frac{\lambda_1}{p_1}  + \frac{\lambda_2}{p_2}
< 1.
\end{equation}
\end{theorem}
The proof for a more general system is given in Appendix A. 
Since both queues have memoryless channels, for any received packet
from queue-$i$, as soon as the server switches to queue $i$, the
expected time to ON state is $1/p_i$. Namely, the time to ON state
is a geometric random variable with parameter $p_i$. Hence, the
effect of i.i.d. connectivity is such that this geometric random
variable is essentially the ``service time per packet'' for
queue-$i$. Note that we call the term $\lambda_1/p_1+\lambda_2/p_2$
the system load, $\rho$, since it is the rate with which the work is
entering the system in the form of service slots.
%
%
In a multiuser single-server system \emph{with or without switchover
times}, with stationary arrivals whose average arrival rates are
$\lambda_i, i \in \{1,2\}$, and i.i.d. service times independent of
the arrivals with average service times $1/p_i, i \in \{1,2\}$, a
necessary condition for stability is given by the system load,
$\rho$, less than 1. To see this, the stability region of the
polling system with zero switchover times is an upperbound on the
stability region of the corresponding system with nonzero switchover
times. Finally, a necessary condition for the stability of the
former system is $\rho = \lambda_1/p_1 + \lambda_2/p_2 <1$, (e.g.,
\cite{Walrand}).
%
Next we show that the stability condition in (\ref{eq:Stab_Necess})
is also sufficient.\\
\emph{Gated Policy:}\\ \emph{Serve all the packets that are present
at a queue upon arrival at the queue.}
\begin{theorem}\label{thm:iid_suff}
{Gated policy together with cyclic order of service for the server
stabilizes the system as long as $\rho < 1$.}
\end{theorem}
The proof for a more general system is given in Appendix B. It is
based on a Lyapunov stability argument over a cycle duration.
Namely, we let $m$ be the discrete time index for the $m$th time the
server stops for servicing a queue and let $T_m$ be the time slot
number of this server-queue meeting times 
Let $I(m)$ be the i.d. of the node that the mobile serves at time
$T_m$ and let $S(Q_{I(m)}(T_m))$ be the service time required to
serve $Q_{I(m)}(T_m)$ packets at time $T_{m}$. Under Gated service
the server serves all $Q_{I(m)}(T_m)$ messages, therefore,
$S(Q_{I(m)}(T_m))$ is the summation of $Q_{I(m)}(T_m)$ independent
geometric random variables of parameter $p_{I(m)}$.
We have the following queue evolution:
\begin{equation}\label{eq:queue evol}
\displaystyle \sum_{i=1}^2 \!Q_{i}(T_{m+1}) \!=\!\! \sum_{i=1}^2
\!Q_{i}(T_m) +\! \sum_{i=1}^2 \!\!
\!\sum_{\,\,\,\tau=T_m}^{\;T_{m+1}-1} \!\!\!\!\!A_{i}(t) \! -
\!Q_{I(m)}(T_m)
\end{equation}
where $T_{m+1}= S(Q_{I(m)}(T_m))+ 1$ with additional $1$ due to
switchover delay. We use the following linear Lyapunov function:
\begin{equation}\label{eq:gated_lyap}
\displaystyle L(\mathbf{Q}(T_{m})) = \sum_{i=1}^2
\frac{Q_i(T_{m})}{p_i}.
\end{equation}
Given the current queue sizes, this Lyapunov function represents the
expected amount of service slots needed to serve the packets present
in both queues. We define the drift over one cycle as
\begin{equation*}
\Delta(T_m) \triangleq E \left\{ L(\mathbf{Q}(T_{m+2})) -
L(\mathbf{Q}(T_{m}))|\mathbf{Q}(T_m)\right\}.
\end{equation*}
Using (\ref{eq:queue evol}) and (\ref{eq:gated_lyap}) one can show
that the drift over the cycle is negative if
\begin{equation}\label{eq:temp7}
\displaystyle \sum_{i=1}^2 \frac{Q_{i}(T_m)}{p_i} >
\rho\frac{2}{1-\rho}.
\end{equation}
To understand the intuition behind this condition, 
first note that $\frac{2}{1-\rho}$ is the expected cycle time in the
system in steady state (in general the expected cycle time is the
total travel time per cycle divided by $1-\rho$) \cite{Takagi}.
Hence, $\rho\frac{2}{1-\rho}$ denotes the expected increase in
system work load over one cycle. Therefore, (\ref{eq:temp7}) argues
that if $\sum_{i=1}^2 \frac{Q_{i}(T_m)}{p_i}$, a lower bound on the
expected decrease in system work load over one cycle, is greater
than the expected increase in system load over one cycle, then the
system is stable. Therefore, the throughput region of the system is
given by 
\begin{equation}\label{eq:Stab}
\mathbf{\Lambda} = \Big \{(\lambda_1,\lambda_2) \big|\;\;
\frac{\lambda_1}{p_1} + \frac{\lambda_2}{p_2}  \le 1 \Big\}.
\end{equation}
\begin{figure}
\centering \psfrag{1\r}[l][][0.95]{$0.5$}
\psfrag{2\r}[l][][0.95]{$\!\!\!\!0.5$}
\psfrag{3\r}[l][][1]{$\lambda_2$} \psfrag{4\r}[l][][1]{$\lambda_1$}
\psfrag{6\r}[l][][0.85]{$\!\!\!\!\!\!\!\!\!\!\!\!\!\!\!\!\!$\textrm{no-switchover
time}} \psfrag{9\r}[l][][0.85]{\textrm{i.i.d.}}
\psfrag{5\r}[l][][0.85]{$\!\!\!\!\!$\textrm{channels}}
\psfrag{7\r}[l][][.85]{$\!$\textrm{Markovian}}
\psfrag{8\r}[l][][.85]{$\!$\textrm{channels}}
\includegraphics[width=0.23\textwidth]{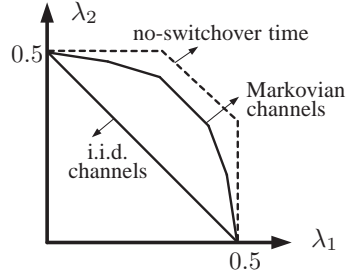}
\vspace{-2mm} \captionsetup{font=footnotesize}\caption{Stability
region under uncorrelated (i.i.d.) and correlated (Markovian)
channels with and without switchover
time.}\label{Fig:stab_region_two_links}
\end{figure}
For the case of two parallel queues, the throughput region of the
system without switchover delay analyzed in \cite{tass93},
$\mathbf{\Lambda}_{ns}$, is given by
\begin{equation}\label{eq:stab_tass}
\mathbf{\Lambda}_{ns} \!\!=\!\! \Big\{(\lambda_1,\lambda_2) \big |
\lambda_1 \le p_1, \lambda_2 \le p_2, \lambda_1 + \lambda_2 \le p_1
+ p_2(1-p_1)\Big \}.
\end{equation}
These two regions are displayed in Fig. \ref{Fig:stab_region_two_links} for the case of
$p_1=p_2=0.5$. 
The stability region of the system without switchover time shrinks
considerably when there is switchover delay. Note that for the case
of deterministic channels, ($C_1(t)=C_2(t)=1, \forall t$), the
systems with or without switchover times have the same stability
region $\lambda_1+\lambda_2 < 1$ \footnote{Throughput region of a
general N-queue polling system with stationary arrivals of rates
$\lambda_i,i \in {1,...,N}$, i.i.d. service processes of mean
service times $s_1,...,s_N$ and finite travel times $D_{ij}$ between
queues $i$ and $j$ is given by $\rho=\sum_{i=1}^N \lambda_i s_i<1$
(see e.g., \cite{Takagi}). Therefore, in the absence of random
connectivity, finite travel times do not affect the stability
region. To see this, considering the system under the optimal Gated
Policy, with arrival rates close to the boundary of the stability
region, the fraction of times the server spends receiving packets
dominates the fraction of time spent on travel.}. Therefore, there
is a significant throughput loss due to switchover delay when the
channel processes are i.i.d. over time. Therefore, \emph{it is the combination of switchover delay and
random connectivity that result in fundamental changes in system
behavior}. 

\begin{remark}
Note that the results of this section hold for more general systems;
namely, for systems with $N$ queues and arbitrary
switchover times between the queues (switchover time from some
queue-$i$ to queue-$j$ given by a constant $d_{ij}\ge 1$ slots). The
stability region in this case is given by
$\lambda_1/p_1+...+\lambda_N/p_N \le 1$.
\end{remark}

With Markovian channels, it is clear that one can achieve better
throughput region than the i.i.d. channels case if the channels are
positively correlated over time. This is because we can exploit the
channel diversity when the channel states stay the same with high
probability. 
In the following, we show that indeed the throughput region
approaches the throughput region of no switchover time case in in
\cite{tass93} as the channels become more correlated over time. Note
that the throughput region in \cite{tass93} is the same for both
i.i.d. and Markovian channels under the condition that probability
of ON state for the $i.i.d.$ channels is the same as the steady
state probability of ON state for the two state Markovian channels.
This fact can be derived as a special case of the seminal work of
Neely in \cite{neely05}.

\section{Stability Region - Correlated Channels}\label{Sec:Stab_Reg}

In this and the following sections we analyze the system under
correlated channels assumption. Assume the channel processes
$C_1(t)$ and $C_2(t)$ are according to Definition
\ref{def:uncorr_ch}. We analytically derive an upper bound on the
throughput region of the system (necessary conditions on $\lambda_1$
and $\lambda_2$ for stability) via analyzing the corresponding
system with saturated queues. As we show in Section
\ref{Sec:Opt_Pol}, the necessary conditions derived in this section
are also sufficient and hence the region established in this section
is the throughput region of the system.

When switchover times are non-zero, channel
correlation impacts the stability region considerably. 
In particular, channel correlation can be exploited to improve the
throughput of the system. Moreover, as $\epsilon \rightarrow 0$, the
stability region tends to that achieved by the system with
no-switchover time and for $0<\epsilon<0.5$ it lies between the
stability regions corresponding to the two extreme cases $\epsilon
=0.5$ and $\epsilon \rightarrow 0$ as  shown in
Fig.~\ref{Fig:stab_region_two_links}.

We start by analyzing the corresponding system with saturated
queues, i.e., both queues are always non-empty. Let
$\mathbf{\Lambda}_s$ denote the set of all time average expected
departure rates that can be obtained from the two queues in the
saturated system under all possible policies that are possibly
history dependent, randomized and non-stationary. We will show that
$\mathbf{\Lambda}=\mathbf{\Lambda}_s$. We prove the necessary
stability conditions 
in the following Lemma and establish sufficiency in the next section.    

\vspace{0.1cm}
\begin{lemma}\label{lem:Lam_under_Lam_s} We have
\begin{equation*}
\mathbf{\Lambda} \subseteq \mathbf{\Lambda}_s.
\end{equation*}
\end{lemma}
\begin{proof}
Given a policy $\pi$ for the original system specifying the switch
and stay actions based possibly on observed channel and queue state
information, consider the saturated system with \emph{the same
sample path of channel realizations} for $t\in \{0,1,2,... \}$ and
\emph{the same set of actions} as policy $\pi$ at each timXe slot
$t\in \{0,1,2,... \}$. Let this policy for the saturated system be
$\pi'$. Let $D_i(t),i \in \{1,2\}$ be total number departures by
time $t$ from queue-$i$ in the original system under policy $\pi$
and let $D_i'(t),i \in \{1,2\}$ be the corresponding quantity for
the saturated system under policy $\pi'$. It is clear that $\lim_{t
\rightarrow \infty} (D_1(t)+D_2(t))/t \le 1$, where the same
statement also holds for the limit of $D_i'(t),i \in \{1,2\}$. Since
some of the ON channel states are wasted in the original system due
to empty queues, we have
\begin{equation}\label{eq:sat_dept_bound}
D_1(t) \le D_1'(t),\;\;\;\textrm{and},\;\;\;D_2(t)\le D_2'(t).
\end{equation}
Therefore, the time average expectation of $D_i(t),i \in \{1,2\}$ is
also less than or equal to the time average expectation of
$D_i'(t),i \in \{1,2\}$. This completes the proof since
(\ref{eq:sat_dept_bound}) holds under any policy $\pi$ for the
original system.
\end{proof}

Now, we derive the region $\mathbf{\Lambda}_s$ by formulating the
system dynamics as a Markov Decision Process (MDP). Let
$\mathbf{s}_t=(m(t),C_1(t),C_2(t)) \in S$ denote the system state at
time $t$ where $S$ is the set of all states. Also, let $a_t \in
A=\{0,1\}$ denote the action taken at time slot $t$ where $A$ is the
set of all actions at each state. 
Let $\mathbb{H}(t)=\mathbf{C}(\tau)|_{\tau=0}^{t}$ denote the full
history of the channel processes until time $t$. For a saturated
system, a policy is a mapping from $\mathbb{H}(t)$ to the set of all
probability distributions on actions $a_t \in \{0,1\}$. This
definition includes randomized policies that choose $a_t$ randomly
at a given state $\mathbf{s}_t$. A \emph{stationary} policy is a
policy that depends only on the current state. In each time slot
$t$, the server observes the current state $\mathbf{s}_t$ and
chooses an action $a_t$. Then the next state $j$ is realized
according to the transition probabilities $\mathbf{P}(j|s,a)$, which
depend on the random channel processes. Now, we define the reward
functions as follows: \beqn
r_1(s_t,a_t) \!\!\!\!\!&=& \!\!\!\!\!1 \textrm{ if } s_t\!\!=\!\!(1,1,1) \textrm{ or } s_t\!\!=\!\!(1,1,0) \textrm{, and } a_t\!\!=\!\!1 \label{eq:dept_rate1} \\
r_2(s_t,a_t) \!\!\!\!\!&=&\!\!\!\!\! 1 \textrm{ if }
s_t\!\!=\!\!(2,1,1) \textrm{ or } s_t\!\!=\!\!(2,0,1) \textrm{, and
} a_t\!\!=\!\!1, \label{eq:dept_rate} \eeqn and
$r_1(s_t,a_t)=r_2(s_t,a_t)=0$ otherwise. That is, a reward is
obtained when the server stays at an ON channel. We are interested
in the
set of all possible time average expected departure rates, 
therefore, given some
$\alpha_1,\alpha_2  \ge 0$, 
define the system reward at time $t$ as
$r(s_t,a_t) = \alpha_1r_1(s_t,a_t) + \alpha_2
r_2(s_t,a_t)$. 
The average reward of policy $\pi$ is defined as
\begin{equation*}
r^{\pi} = \displaystyle \lim_{K\rightarrow \infty} \frac{1}{K} E
\Big \{ \sum_{t=1}^K r(s_t,a_t^{\pi}) \Big \}.
\end{equation*}
 Given some $\alpha_1,\alpha_2 \ge 0$, we are interested in the policy that
achieves the maximum time average expected reward
$r^{*} = \max_{\pi} r^{\pi}$. 
This optimization problem is a discrete time 
MDP characterized by the state transition probabilities
$\mathbf{P}(j|s,a)$ with 8 states and 2 actions per state.
Furthermore, under every policy, the underlying Markov chain that
describes the system state evolution has a single recurrent class
plus possibly a set of transient states. Note that we eliminate the
policy that switches in all 8 states and achieves 0 total average
rate. Therefore this MDP belongs to the class of \emph{Unichain}
MDPs \cite{puterman05}.
%
%
%
%
%
For Unichain MDPs with finite state and action spaces, we can define
the \emph{state-action polytope}, $\mathbf{X}$, as the set of
16-dimensional vectors $\mathbf{x}$ that satisfy the balance
equations \beqn x(s;1) + x(s;0) = \!\!\sum_{s'}
\!\!\sum_{a\in\{0,1\}}\!\!\! \mathbf{P}\big(s|s',a\big)x(s';a) ,
\;\forall\; s \in S, \label{eq:saf_balance} \eeqn the normalization
condition \beqn \sum_{s}x(s;1) + x(s;0)=1,  \label{eq:saf_norm}
\eeqn and the nonnegativity constraints \beqn x(s;a)\ge 0,  \:
\mbox{for} \: s \in S, a\in A. \label{eq:saf_nonneg} \eeqn Note that
$x(s;1)$
can be interpreted as the stationary
probability that action
\emph{stay} 
is taken at state $s$. 
More precisely, a point $\mathbf{x}\in\mathbf{X}$ 
corresponds to a randomized stationary policy 
that takes action $a \in \{0,1\}$ 
at state $s$ w.p.
\begin{equation}\label{eq:rand_stat}
\,\!\mathbf{P}(\textrm{\emph{action $a$} at state$\,s$})\!=\!
\frac{x(s;a)}{x(s;1)+x(s;0)}, a \in A,s\in S_x, \!\!\!\!\!\!\!\!\!
\end{equation}
where $S_x$ is the set of 
recurrent states  given by
$S_x\equiv \displaystyle \{ s\in S: x(s;1)+x(s;0)>0\},$
and actions are arbitrary for transient states $s \in S/S_x$ \cite[Theorem 8.8.6]{puterman05}. 
Furthermore, every policy 
has a unique limiting average state action
frequency 
in $\mathbf{X}$ regardless of the initial
state distribution  \cite[Theorem 8.9.3]{puterman05}. Therefore,
given \emph{any} policy, there exists a stationary randomized policy
with the same
limiting state action frequencies \cite{puterman05}. 
%
%
%
These facts imply that when searching for the optimal policies, one
can restrict attention to stationary randomized policies as in
(\ref{eq:rand_stat}) for
$\mathbf{x}\in\mathbf{X}$. 

The following linear transformation of the state-action polytope
$\mathbf{X}$ defines 
the \emph{reward
polytope} \cite{shie05}:
%
%
$\{(\overline{r}_1,\overline{r}_2)\big| \overline{r}_1=
\mathbf{x}.\mathbf{r}_1, \overline{r}_2= \mathbf{x}.\mathbf{r}_2,
\mathbf{x} \in \mathbf{X}\}$, where $(.)$ denotes the vector inner
product and $\mathbf{r}_1$ and $\mathbf{r}_2$ are the 16-dimensional
reward functions defined in (\ref{eq:dept_rate1}) and
(\ref{eq:dept_rate}). This polytope is the set of all time average
expected departure rate pairs that can be obtained in the saturated
system, i.e., it is the rate region $\mathbf{\Lambda}_s$. An
explicit way of deriving
$\mathbf{\Lambda}_s$ is given in Algorithm \ref{alg:LP_for_Lambda_s}. 

\begin{algorithm}[h]
\caption{\emph{Stability Region Characterization}}
\label{alg:LP_for_Lambda_s}
\begin{algorithmic}[1]

\STATE Given $\alpha_1,\alpha_2 \ge 0$, solve the following Linear
Program
\begin{eqnarray}
\max_{x}  \quad  \alpha_1 \overline{r}_1 +  \alpha_2 \overline{r}_2  \nonumber \\
\mbox{subject to} \quad  \mathbf{x} \in \mathbf{X}. \label{eq:LP}
\end{eqnarray}

\STATE For a given $\alpha_2/\alpha_1$ ratio, the optimal solution
$(\overline{r}_1^{*}, \overline{r}_2^{*})$ of the LP in
(\ref{eq:LP}) gives one of the corner points of 
$\mathbf{\Lambda_s}$. Find all possible corner points and take their
convex combination.
\end{algorithmic}
\end{algorithm}
%
%
%
%
The following lemma is useful for finding the solutions of the above
LP for all possible $\alpha_2/\alpha_1$ ratios 
\cite[Corollary~8.8.7]{puterman05}.
%
\begin{lemma}
\emph{Suppose $\mathbf{x}$ is a vertex for the LP in (\ref{eq:LP}),
then the stationary randomized policy corresponding to $\mathbf{x}$,
as defined in (\ref{eq:rand_stat}), is a deterministic policy.
Conversely, for any stationary deterministic policy, the stationary
distribution of states induced by the policy is a vertex for the LP
in (\ref{eq:LP})}.
\end{lemma}
\vspace{0.1cm}

The intuition behind this lemma is as follows. For simplicity assume
all states are recurrent. 
Note that the more general case can be argued similarly. 
Now suppose $\mathbf{x} \in \mathbf{X}$, i.e.,
$\mathbf{x}(s,a) \ge 0,\, \mathbf{s}\in S,\, a\in A$, and
$\mathbf{x}$ satisfies all the equality constraints in
(\ref{eq:saf_balance}) and (\ref{eq:saf_norm}) out of which only 8
are linearly independent. For a 16 dimensional vector $\mathbf{x}
\in \mathbf{X}$ to be a vertex, we need to have at least 16 linearly
independent active constraints at $\mathbf{x}$.
%
%
%
%
%
If $\mathbf{x}$ corresponds to a deterministic policy, then either
$\mathbf{x}(s,1)$ or $\mathbf{x}(s,0)$ has to be zero. This gives at
least $|S|=8$ more linearly independent active constraints at
$\mathbf{x}$, satisfying the vertex condition.
%

Therefore, \emph{the corners of the rate polytope
$\mathbf{\Lambda_s}$ are given by stationary deterministic
policies}. There are a total of $2^8$ stationary deterministic
policies since we have $8$ states and $2$ actions per state. Hence,
finding the rate pairs corresponding to the 256 deterministic
policies and taking their convex combination gives
$\mathbf{\Lambda}_s$. Fortunately, we do not have to go through this
tedious procedure. The fact that at a vertex of (\ref{eq:LP}) either
$x(s;1)$ or $x(s;0)$ has to be zero for each $\mathbf{s} \in
\mathbf{S}$ provides a useful guideline for analytically solving
this LP. 
The following theorem, proved in Appendix C, is based on this solution to find the corners
of $\mathbf{\Lambda_s}$ and then applying
Algorithm~\ref{alg:LP_for_Lambda_s}. 
It is one of key results of this paper
characterizing the stability region explicitly.
\vspace{0.1cm}
\begin{theorem}\label{thm:stab}
\emph{The rate region $\mathbf{\Lambda}_s$ is the set of all arrival
rates $\lambda_1\ge 0$, $\lambda_2 \ge 0$ that for $\epsilon <
\epsilon_c = 1-\sqrt{2}/2$ satisfy} \beqn
\epsilon\lambda_1 + (1-\epsilon)^2\lambda_2 &\le& \frac{(1-\epsilon)^2}{2}\nonumber\\
(1-\epsilon)\lambda_1 + (1+\epsilon-\epsilon^2)\lambda_2  &\le& \frac{3}{4}-\frac{\epsilon}{2}\nonumber\\
\lambda_1 + \lambda_2  &\le& \frac{3}{4}-\frac{\epsilon}{2}\nonumber\\
(1+\epsilon-\epsilon^2)\lambda_1 + (1-\epsilon)\lambda_2  &\le& \frac{3}{4}-\frac{\epsilon}{2}\nonumber\\
(1-\epsilon)^2\lambda_1 + \epsilon\lambda_2 &\le&
\frac{(1-\epsilon)^2}{2},\nonumber 
\eeqn
 \emph{and for $\epsilon \ge \epsilon_c $ satisfy}
\beqn
\lambda_1 + (1-\epsilon)(3-2\epsilon)\lambda_2 &\le& \frac{(1-\epsilon)(3-2\epsilon)}{2}\nonumber\\
\lambda_1 +\lambda_2  &\le& \frac{3}{4}-\frac{\epsilon}{2}\nonumber\\
(1-\epsilon)(3-2\epsilon)\lambda_1 + \lambda_2  &\le&
\frac{(1-\epsilon)(3-2\epsilon)}{2}.\nonumber
 \eeqn
\end{theorem}

The stability regions for these two ranges of $\epsilon$ are
displayed in Fig.~\ref{Fig:throughput_eps_0_250_write_up}~(a) and
(b). As $\epsilon \rightarrow 0.5$, the stability region converges
to that of the i.i.d. channels with ON probability equal to $0.5$.
In this regime, knowledge of the current channel state is of no
value.
%
%
%
%
As $\epsilon\rightarrow 0$ the stability region converges to that
for the system with no-switchover time in \cite{tass93}. In this
regime, the channels are likely to stay the same in several
consecutive time slots, therefore, the effect of switching delay is
negligible.

\begin{figure}
\begin{center}
\includegraphics[width=0.4\textwidth]{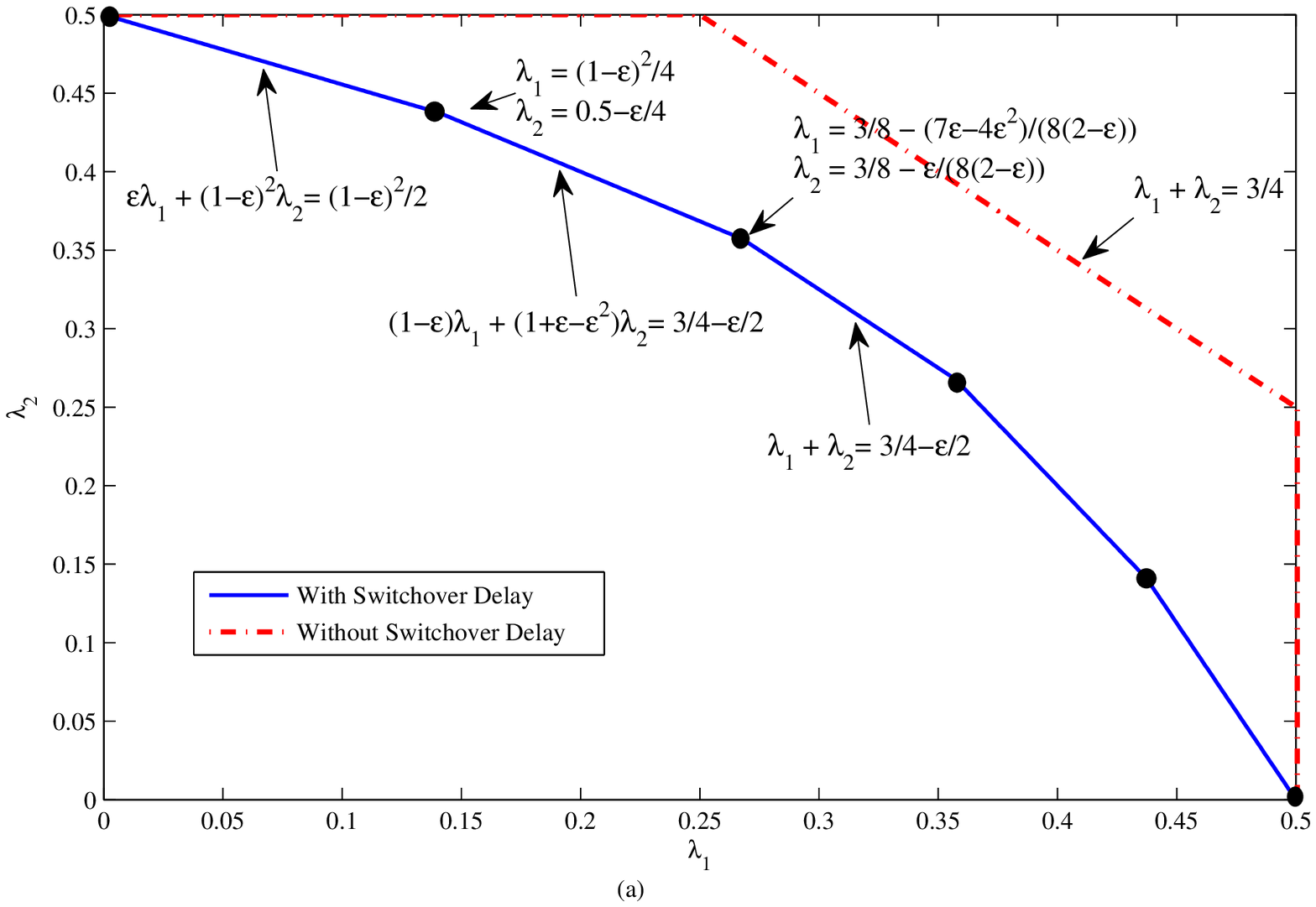}
\includegraphics[width=0.4\textwidth]{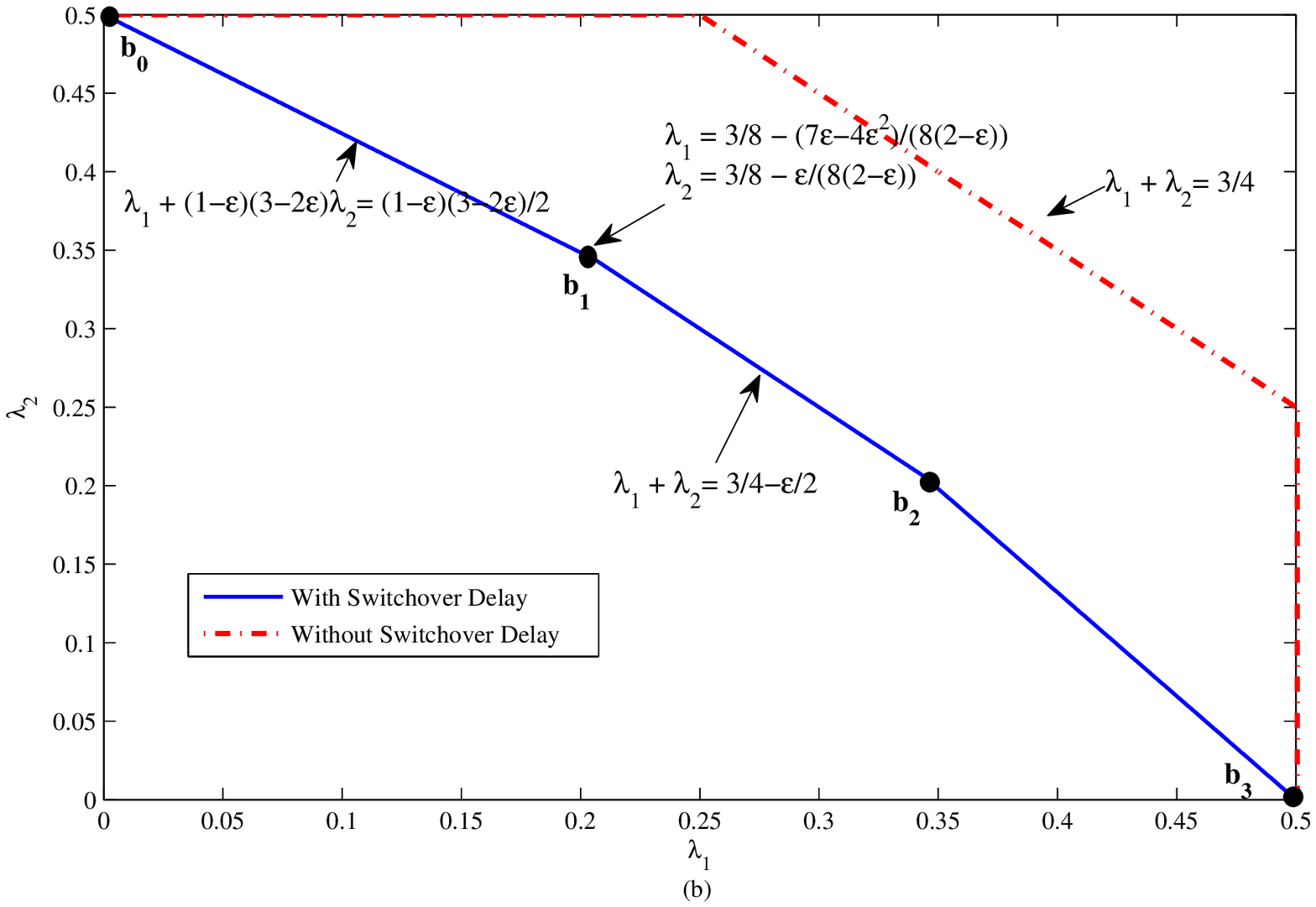}
\end{center}
\vspace{-4mm} \captionsetup{font=footnotesize}\caption{Stability
region under correlated channels with and without switchover time
for (a) $\epsilon=0.25 < \epsilon_c$ and (b) $\epsilon =0.40 \ge
\epsilon_c$.} \label{Fig:throughput_eps_0_250_write_up}
\end{figure}

\begin{remark}
Stability region expressions in terms of the channel parameter
$\epsilon$ are for two parallel queues, two-state Markovian channels
and one slot switching time. However, the technique used for
characterizing the stability region in terms of the state action
frequencies 
is general. For instance, this technique can be used to find the
stability region of systems with more than two queues, arbitrary
switchover times, and more complicated Markovian channel processes.
\end{remark}
%

As we stated before, the corner points of the polytope
$\mathbf{\Lambda}_s$ correspond to deterministic policies. We
observe that the 4 corner points of the rate region polytope for
$\epsilon \ge \epsilon_c$ case are also present in the rate region
polyopte for $\epsilon < \epsilon_c$ case. One of the two additional
corners for the $\epsilon<\epsilon_c$ case, namely, the policy for
the corner point $(r_1,r_2)=((1-\epsilon)^2/4,(2-\epsilon)/4)$,
corresponds to the following deterministic policy: At queue-1: stay
only at $(C_1,C_2)=(1,0)$ and switch in all other states and at
queue-2:
switch only at $(C_1,C_2)=(1,0)$ and stay at all other states. 
The critical point is that this policy decides to switch at queue-1
when the channels are $(C_1,C_2)=(1,1)$, which does not provide a
vertex for the rate region for $\epsilon \ge \epsilon_c$.
Analytically, this is because
$(r_1,r_2)=((1-\epsilon)^2/4,(2-\epsilon)/4)$ is within the convex
combination of the 4 vertices of the rate region for $\epsilon\ge
\epsilon_c$. The intuitive reason behind this is that as $\epsilon$
increases, the predictions about the future channel states become
less reliable. Therefore, as $\epsilon$ increases, switching at
$(C_1,C_2)=(1,1)$ at queue-1 for future ON channel states at queue-2
becomes less preferable than a successful transmission available at
queue-1 in the current slot.
%

\section{Frame Based Dynamic Control Policy}\label{Sec:Opt_Pol}

We propose a frame-based dynamic control (FBDC) policy inspired by
the 
state action frequencies and prove that it is throughput-optimal
asymptotically in the frame length. The motivation behind the FBDC
policy is that a policy $\pi^*$ that achieves the optimization in
(\ref{eq:LP}) for given weights $\alpha_1$ and $\alpha_2$ for the
saturated system, should achieve a \emph{good} performance also in
the original system when
the queue sizes $Q_1$ and $Q_2$ are used as weights. 
This is because first, the policy $\pi^*$ will lead to similar
average departure rates in both systems for sufficiently high
arrival rates, and second, the usage of queue sizes as weights
creates self adjusting policies that capture the dynamic changes due
to stochastic arrivals. This is similar to 
the structure of the celebrated max-weight scheduling in
\cite{tass92}. Specifically, divide the time into equal-size
intervals of $T$ slots and let $Q_1(jT)$ and $Q_2(jT)$ be the queue
lengths at the beginning of the $j$th interval. We find the
deterministic policy that optimally solves (\ref{eq:LP}) when
$Q_1(jT)$ and $Q_2(jT)$ are used as weights 
and then apply this policy in each time slot of the
frame. The FBDC policy is described below 
in details.

\begin{algorithm}[h]
\caption{\textsc{Frame Based Dynamic Control (FBDC) Policy}}
\label{alg:FBDC}
\begin{algorithmic}[1]

\STATE Find the optimal solution to the following Linear Program
\beqn
\displaystyle \mbox{max.}_{\left\{r_1, r_2\right\}} &  Q_1(jT) r_1 + Q_2(jT) r_2 \nonumber \\
\mbox{subject to}  &   \left( r_1, r_2\right) \in \mathbf{\Lambda}_s
\label{objcon}
\eeqn where $\mathbf{\Lambda}_s$ is the rate polytope derived in
Section \ref{Sec:Stab_Reg}. \STATE The optimal solution $(r_1^{*},
r_2^{*})$ in step 1 is a corner point of $\mathbf{\Lambda}_s$ that
corresponds to a stationary deterministic policy denoted by $\pi^*$.
Apply $\pi^*$ in each time slot of the frame.
\end{algorithmic}
\end{algorithm}

\begin{theorem}\label{thm:FBDC}
\emph{The FBDC policy stabilizes the system as long as the arrival
rates $(\lambda_1,\lambda_2)$ are within the
$\mathbf{\delta}$-stripped stability region
$\mathbf{\Lambda_s^{\delta}}$ where $\delta(T)$ is a decreasing function of $T$.}
\end{theorem}
The proof is given in Appendix D. It performs a drift
analysis using the standard quadratic Lyapunov function. However, \emph{it
is novel in utilizing an MDP framework in Lyapunov
drift arguments}. The basic idea is that when the
optimal policy solving (\ref{objcon}), $\pi^*$,
is applied over
a sufficiently long frame of $T$ slots, the average 
output rates of both the actual system and the corresponding saturated
system converge to $\mathbf{r}^*$. 
For the saturated system, the difference between empirical rates and
$\mathbf{r^*}$ is essentially due to the convergence of the Markov chain 
induced by policy $\pi^*$ to its steady state,
which is exponentially fast in $T$
\cite{shie05}. Therefore, for sufficiently large
queue lengths, the difference between the empirical rates in the
actual system and $r^*$ 
also decreases with $T$. This
ultimately results in a negative Lyapunov drift when $\boldsymbol \lambda$ is
inside the $\delta(T)$-stripped stability region
since from (\ref{objcon}) we have $\mathbf{Q}(jT).\mathbf{r}^* \ge
\mathbf{Q}(jT).\boldsymbol \lambda$.

The parameter $\delta(T)$, capturing the difference between the stability region of the FBDC policy and $\mathbf{\Lambda}_s$, is related to the mixing time of the system Markov chain and is a decreasing function of $T$. This establishes that the FBDC policy is asymptotically optimal and that $\mathbf{\Lambda}=\mathbf{\Lambda_s}$. Moreover, as also suggested by the simulation results in Section \ref{Sec:Sim}, $\delta(T)$ is negligible even for relatively small values of $T$.
%
%

The FBDC policy is easy to implement since it does not require the
solution of the LP for each frame. Instead, we can solve the LP for
all possible $Q_2(t)/Q_1(t)$ values only \emph{once} in advance and
create a mapping from the $Q_2(t)/Q_1(t)$ values to the corner
points of the stability region. Then, we can use this mapping to
find the corresponding optimal policy $\pi^*$ at the beginning of
each frame. Such a mapping depends only on the \emph{slopes} of the
lines in the stability region in
Fig.~\ref{Fig:throughput_eps_0_250_write_up}. Therefore, these mappings are already
available  and are given in figures
\ref{Fig:optimal_map} and \ref{Fig:optimal_map2}.
%
%

\begin{remark}
FBDC policy is a \emph{generic} policy applicable to much more
general systems. For instance, it provides throughput optimality for
systems with more number of queues, with swithcover time from some
queue-$i$ to queue-$j$ given by a constant $d_{ij}\ge 1$ slots,
and with more complicated 
Markov modulated channel structures. FBDC can also be used to
achieve stability for classical network control problems such the
one with no-switchover times analyzed in \cite{tass93}. 
\end{remark}
\begin{remark}
FBDC policy provides a new framework for developing
throughput-optimal policies for network control. Namely, given any
queuing system whose corresponding saturated system is Markovian
with a finite state space, throughput optimality is easily achieved
by solving an LP in order to find the stationary MDP solution of the
corresponding saturated system and applying this solution over a
frame in the actual system.
\end{remark}

\begin{figure}
\centering
\psfrag{1\r}[l][][.8]{\!\!\!\!\!\!\!$\frac{\epsilon}{(1-\epsilon)^2}$}
\psfrag{2\r}[l][][.8]{\!\!\!\!\!\!\!$\frac{1-\epsilon}{1+\epsilon-\epsilon^2}$}
\psfrag{3\r}[l][][.8]{\!\!$1$}
\psfrag{4\r}[l][][.8]{\!\!\!\!\!\!\!$\frac{1+\epsilon-\epsilon^2}{1-\epsilon}$}
\psfrag{5\r}[l][][.8]{\!\!\!\!\!\!\!$\frac{(1-\epsilon)^2}{\epsilon}$}
\psfrag{6\r}[l][][.8]{$b_5$} \psfrag{7\r}[l][][.8]{$b_4$}
\psfrag{8\r}[l][][.8]{$b_3$} \psfrag{9\r}[l][][.8]{$b_2$}
\psfrag{a\r}[l][][.8]{$b_1$} \psfrag{b\r}[l][][.8]{$b_0$}
\psfrag{c\r}[l][][1]{$\frac{Q_2}{Q_1}$}
\includegraphics[width=0.35\textwidth]{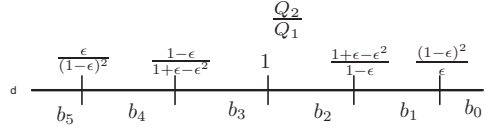}
\vspace{-1mm}\caption{Mapping from the queue sizes to the capacity
region corner points for the FBDC policy for
$\epsilon<\epsilon_c$.}\label{Fig:optimal_map}
\end{figure}

\begin{figure}
\centering
\psfrag{1\r}[l][][.8]{\!\!\!\!\!\!\!\!\!\!$\frac{1}{(1-\epsilon)(3-2\epsilon)}$}
\psfrag{3\r}[l][][.8]{\!\!$1$}
\psfrag{5\r}[l][][.8]{\!\!\!\!\!\!\!\!\!\!$(1-\epsilon)(3-2\epsilon)$}
\psfrag{6\r}[l][][.8]{$b_5$} \psfrag{7\r}[l][][.8]{$b_3$}
\psfrag{8\r}[l][][.8]{$b_2$} \psfrag{9\r}[l][][.8]{$b_0$}
\psfrag{c\r}[l][][1]{$\frac{Q_2}{Q_1}$}
\includegraphics[width=0.35\textwidth]{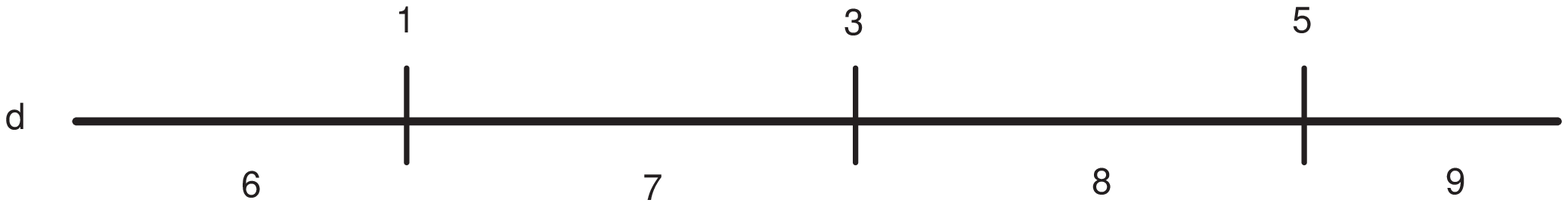}
\vspace{-1mm}\caption{Mapping from the queue sizes to the capacity
region corner points for the FBDC policy for
$\epsilon\ge\epsilon_c$.}\label{Fig:optimal_map2}
\end{figure}

Note that the FBDC policy does not require the knowledge of the arrival rates, the channel statistics or the capacity region. The mapping in figures \ref{Fig:optimal_map} and \ref{Fig:optimal_map2} is given in terms of threshold on $\epsilon$ since for two queues these values are available. For general networks of many queues and arbitrary switchover times, the corresponding table of mappings from the queue sizes to stationary deterministic policies can be obtained by solving the LP in (\ref{eq:LP}) using the queue sizes $(Q_1(jT),Q_2(jT)$ as weights.

In the next section we consider Myopic policies that do not require
the solution of an LP and provide stability for more than
$96\%$ of the stability region. 
Simulation results in Section \ref{Sec:Sim} suggest that the Myopic
policies may indeed achieve the full stability region while
providing better delay performance than
the FBDC policy for most arrival rates. 

\section{Myopic Control Policies}\label{Sec:Myopic_Pol}
\vspace{-0.05mm}

Next, we investigate the performance of simple \emph{Myopic}
policies. We implement these policies in a frame-based fashion where
the scheduling/switching decisions during a frame of $T$ time slots
are based on queue lengths at the beginning of the frame and channel
predictions for a small number of slots into the future. We refer to
a Myopic policy considering $k$ future time slots as the
$k$-Lookahead Myopic policy. In the $1$-Lookahead Myopic policy, the
server chooses the queue with the larger weight where the weight of
a queue is the product of the queue length and the expected number
of departures in the current and the next slot
from the queue. 
%
%
The detailed description of the $1$-Lookahead Myopic policy is given
below. 
%
\vspace{-3.3mm}
\begin{algorithm}[h]
\caption{\textsc{$1$-Lookahead Myopic Policy}} \label{alg:Myop1}
\begin{algorithmic}[1]

\STATE Assuming that the server is currently with queue 1 and the
system is at the $j$th frame, calculate the following weights in
each time slot of the current frame; \beqn
W_1(t) \!\!\!\!&=\!\!\!\!&
Q_1(jT)\Big(C_1(t)+\mathbb{E}\big[C_1(t+1)|C_1(t)\big]\Big)\nonumber\\
W_2(t) \!\!\!\!&=\!\!\!\!&
Q_2(jT)\mathbb{E}\big[C_2(t+1)|C_2(t)\big].\label{eq:myopic}
\eeqn \STATE
If $W_1(t)\ge W_2(t)$ stay with queue one, 
otherwise, switch to the other queue. A similar rule apply for queue
2. 
%
\end{algorithmic}
\end{algorithm}
\vspace{-2.6mm}

Next we establish a lower bound on the stability region of the
1-Lookahead Myopic Policy 
by comparing its drift over a frame to the drift of the FBDC policy.
\begin{theorem}\label{thm:Myopic_1LH}
\emph{The 1-Lookahead Myopic policy achieves at least
$\gamma$-fraction of the stability region $\mathbf{\Lambda_s}$
asymptotically in $T$ where $\gamma \ge 90\%$.}
\end{theorem}
The proof is constructive and will be establish in various steps in the following.
The basic idea behind the proof is
that the 1-Lookahead Myopic policy produces a mapping from the
set of queue sizes to the stationary deterministic policies
corresponding to the corners of the stability region. This mapping
is similar to that of the FBDC policy, however, the thresholds on
the queue size ratios $Q_2/Q_1$ are determined according to
(\ref{eq:myopic}). We first show the derivation of this mapping and
then bound the difference between the weighted average departure
rates of the 1-Lookahead Myopic and the FBDC policies. We refer to the 1-Lookahead Myopic
policy as the Myopic policy in the following. \\
\\
\textbf{Mapping from queue sizes to actions. Case-1: $\epsilon <
\epsilon_c$ }
For each corner of the throughput region, we will find
the range of coefficients $Q_1$ and $Q_2$ such that the Myopic
policy chooses the deterministic actions corresponding to the given
corner. We enumerate the corners of the throughput region as $b_0,
b_1,...,b_5$
where $b_0$ is $(0,0.5)$ and $b_5$ is $(0.5,0)$.\\
\\
\textbf{Corner $b_0$}:\\
Optimal actions are to stay at queue-2 for every channel condition.
Therefore, the server chooses queue-2 even when the channel state is
$C_1(t),C_2(t)=(1,0)$. Therefore, using (\ref{eq:myopic}), for the
Myopic policy to take the deterministic actions corresponding to
$b_0$ we need
\begin{equation*}
Q_1.(1-\epsilon) < Q_2.(\epsilon)\;\;\Rightarrow \frac{Q_2}{Q_1} >
\frac{1-\epsilon}{\epsilon}.
\end{equation*}
This means that if we apply the Myopic policy with coefficients
$Q_1, Q_2$ such that $Q_2/Q_1 > (1-\epsilon)/\epsilon$, then the
system output rate will be driven towards the corner point $b_0$
(both in the saturated system or in the actual system with large enough arrival rates). \\
\textbf{Corner $b_1$}:\\
The optimal actions for the corner point $b_1$ are as follows: At
queue-1, for the channel state $10$:stay, for the channel states
$11$, $01$ and $00$: switch. At queue-2, for the channel state $10$:
switch, for the channel states $11$, $01$ and $00$: stay. The most
limiting conditions are $11$ at queue-1 and $10$ at queue-2.
Therefore we need, $Q_1(2-\epsilon) < Q_2 (1-\epsilon)$ and
$Q_1(1-\epsilon) > Q_2\epsilon$. Combining these we have
\begin{equation*}
 \frac{2-\epsilon}{1-\epsilon}<\frac{Q_2}{Q_1} <
\frac{1-\epsilon}{\epsilon}.
\end{equation*}
Note that the condition $\epsilon < \epsilon_c = 1-\sqrt{2}/2$
implies that $\frac{1-\epsilon}{\epsilon}
> \frac{2-\epsilon}{1-\epsilon}$.\\
\textbf{Corner $b_2$}:\\
The optimal actions for the corner point $b_1$ are as follows: At
queue-1, for the channel state $10$ and $11$:stay, for the channel
states $01$ and $00$: switch. At queue-2, for the channel states
$10$: switch, for the channel states $11$, $01$ and $00$: stay. The
most limiting conditions are $11$ at queue-1 and $00$. Therefore we
need, $Q_1(2-\epsilon) > Q_2 (1-\epsilon)$ and $Q_1 < Q_2$.
Combining these we have
\begin{equation*}
1 <\frac{Q_2}{Q_1} <  \frac{2-\epsilon}{1-\epsilon}.
\end{equation*}
The conditions for the rest of the corners are symmetric and can be
found similarly to obtain the mapping in Fig. \ref{Fig:myopic_map}.\\
\\
\textbf{Mapping from queue sizes to actions. Case-2: $\epsilon \ge
\epsilon_c$ }\\
In this case there are 4 corner points in the
throughput region. We enumerate these corners as $b_0, b_2,b_3,b_5$
where $b_0$ is $(0,0.5)$ and $b_5$ is $(0.5,0)$.\\
\textbf{Corner $b_0$}:\\
The analysis is the same as the $b_0$ analysis in the previous case
and we obtain that for the Myopic policy to take the deterministic
actions corresponding to $b_0$ we need
\begin{equation*}
 \frac{Q_2}{Q_1} >
\frac{1-\epsilon}{\epsilon}.
\end{equation*}\\
\textbf{Corner $b_2$}:\\
This is the same corner point as in the previous case corresponding
to the same deterministic policy: At queue-1, for the channel state
$10$ and $11$:stay, for the channel states $01$ and $00$: switch. At
queue-2, for the channel states $10$: switch, for the channel states
$11$, $01$ and $00$: stay. The most limiting conditions are $10$ at
queue-2 (since $\epsilon \ge \epsilon_c$ we have
$\frac{1-\epsilon}{\epsilon} < \frac{2-\epsilon}{1-\epsilon}$) and
$00$. Therefore we need, $Q_1(1-\epsilon)
> Q_2\epsilon$ and $Q_1 < Q_2$. Combining these we have
\begin{equation*}
1 <\frac{Q_2}{Q_1} <  \frac{1-\epsilon}{\epsilon}.
\end{equation*}
The conditions for the rest of the corners are symmetric and can be
found similarly to obtain the mapping in Fig. \ref{Fig:myopic_map2}
for $\epsilon \ge\epsilon_c$.\\

\textbf{Drift Analysis}\\

In each frame, FBDC policy drives the system output rate towards the
corner point of the throughput region that is the solution of the
optimization in (\ref{objcon}) (i.e., according to the mappings in
figures \ref{Fig:optimal_map} and \ref{Fig:optimal_map2}). 
The Myopic policy performs a similar operation but according to the
different mappings given in figures \ref{Fig:myopic_map} and
\ref{Fig:myopic_map2}. In the following we will analyze in which
$Q_2(t)/Q_1(t)$ regions the Myopic and the FBDC policies drive the
system towards different corner points. We will bound the resulting
difference between weighted output rates of the two policies where
the weights are the queue sizes at the beginning of the frames,
thereby obtaining a worst case performance for the Myopic policy.
Writing the drift expressions for the Myopic policy similar to the
proof of Theorem \ref{thm:FBDC}, we have from
(\ref{eq:drift_comparison}) that \beqn \frac{\Delta_T(t)}{2T} \!\leq
\!\!(B\!+\!2)T \!\!+\!\! \sum_{i} Q_i(t) \lambda_i \!-\!\!\!
\sum_{i} Q_i(t) r_i^{My} \!\!+\! \delta_3\!\!\!\sum_{i}Q_i(t)
\label{eq:drift_My} \eeqn where $\mathbf{r}^{My}$ is the corner
point obtained from one of the Myopic mappings and $\delta_3(T)$ is
a decreasing function of $T$. Now let
\begin{equation*}
W_{My}=\sum_{i} Q_i(t) r_i^{My} + \bigg(\sum_{i}Q_i(t)
\bigg)\delta_3(T)
\end{equation*}
and
\begin{equation*}
W_{FBDC}=\sum_{i} Q_i(t) r_i^{*} + \bigg(\sum_{i}Q_i(t)
\bigg)\delta(T).
\end{equation*}
denote the time average weighted departure rates corresponding to
the two policies. Also denote the ratio of the two by $\Psi_T=
W_{My}/W_{Opt}$. We find lower bounds on the ratio $\Psi$ over all
queue sizes at the beginning of the current frame, which will
constitute a lower bound on the stability region of the Myopic
policy. First, consider the simpler ratio $\Psi'$ given by
\begin{equation*}
\Psi'=\frac{\sum_{i} Q_i(t) r_i^{My}}{\sum_{i} Q_i(t) r_i^{*}}.
\end{equation*}
We later take the $\delta$ factors into account via the argument
that $\Psi'\rightarrow \Psi$. The following lemma is proved in
Appendix E and essentially constitutes a lower bound on the
achievable throughput region of the Myopic policy.
\begin{lemma}\label{lem:My_FBDC_comp}
$\Psi' \ge 0.9002.$
\end{lemma}
\begin{figure}
\centering
\psfrag{1\r}[l][][.8]{$\!\!\!\!\frac{\epsilon}{1-\epsilon}$}
\psfrag{2\r}[l][][.8]{$\!\!\!\!\frac{1-\epsilon}{2-\epsilon}$}
\psfrag{3\r}[l][][.8]{$1$}
\psfrag{4\r}[l][][.8]{$\!\!\!\!\frac{2-\epsilon}{1-\epsilon}$}
\psfrag{5\r}[l][][.8]{$\!\!\!\!\frac{1-\epsilon}{\epsilon}$}
\psfrag{6\r}[l][][.8]{$b_5$} \psfrag{7\r}[l][][.8]{$b_4$}
\psfrag{8\r}[l][][.8]{$b_3$} \psfrag{9\r}[l][][.8]{$b_2$}
\psfrag{a\r}[l][][.8]{$b_1$} \psfrag{b\r}[l][][.8]{$b_0$}
\psfrag{c\r}[l][][1]{$\frac{Q_2}{Q_1}$}
\includegraphics[width=0.35\textwidth]{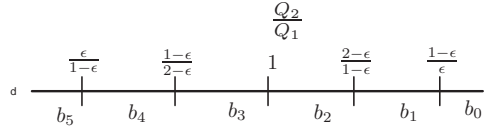}
\vspace{-1mm}\caption{Mapping from the queue sizes to the throughput
region corner points for the 1-Lookahead Myopic policy for
$\epsilon<\epsilon_c$.}\label{Fig:myopic_map}
\end{figure}

\begin{figure}
\centering
\psfrag{1\r}[l][][.8]{$\!\!\!\!\frac{\epsilon}{1-\epsilon}$}
\psfrag{3\r}[l][][.8]{$1$}
\psfrag{5\r}[l][][.8]{$\!\!\!\!\frac{1-\epsilon}{\epsilon}$}
\psfrag{6\r}[l][][.8]{$b_5$} \psfrag{7\r}[l][][.8]{$b_3$}
\psfrag{8\r}[l][][.8]{$b_2$} \psfrag{9\r}[l][][.8]{$b_0$}
\psfrag{c\r}[l][][1]{$\frac{Q_2}{Q_1}$}
\includegraphics[width=0.35\textwidth]{myopic_map2.eps}
\vspace{-1mm}\caption{Mapping from the queue sizes to the throughput
region corner points for the 1-Lookahead Myopic policy for
$\epsilon\ge\epsilon_c$.}\label{Fig:myopic_map2}
\end{figure}

Now consider the expression for $\Psi$ given by
\begin{equation*}
\Psi = \frac{\sum_{i} Q_i(t) r_i^{My} + \bigg(\sum_{i}Q_i(t)
\bigg)\delta_3(T)}{\sum_{i} Q_i(t) r_i^{*} + \bigg(\sum_{i}Q_i(t)
\bigg)\delta(T)}.
\end{equation*}
Since $\delta(T)$ and $\delta_3(T)$ are both decreasing with $T$, we
have that
\begin{equation*}
\Psi \ge \Psi' - \delta_4(T)
\end{equation*}
Choosing $T$ large enough so that $\delta_4(T) \le 0.0002$ we have
that $\Psi \ge \Psi'-0.0002 = 0.9000$ and hence $W_{My} \ge 0.9
W_{FBDC}$. Utilizing this in (\ref{eq:drift_My}) we have the
following for the Myopic policy \beqn \frac{\Delta_T(t)}{2T}
\!\!\leq\!\! (B \!+\! 2)T\!\! +\!\! \sum_{i} \!Q_i(t) \lambda_i\! -
0.9\!\!\sum_{i}\! Q_i(t) r_i^{*} \!+\! \delta_5\sum_{i}Q_i(t),
\nonumber\eeqn where $\delta_5=0.9\delta$ is a very small and
positive number (decreasing with $T$). Now for
$(\lambda_1,\lambda_2)$ strictly inside the 0.9 fraction of the
$\delta_5$-\emph{stripped} stability region,
there exist a small $\vec{\xi} >0$ such that $(\lambda_1,\lambda_2)
+ (\xi,\xi) = 0.9(r_1,r_2) -(\delta_5,\delta_5)$, for some
$\mathbf{r}=(r_1,r_2) \in \mathbf{\Lambda_s}$. Substituting this
expression for $(\lambda_1,\lambda_2)$ and using $\sum_{i} Q_i(t)(r-
r_i^*) \le 0$ we have,
\beqn \frac{\Delta_T(t)}{2T} \leq (B+2)T - \Big(\sum_{i}Q_i(t)
\Big)\xi. \nonumber \eeqn Therefore, the system is stable for
$\lambda$ inside at least the 0.9 fraction of
$\delta_5$-\emph{stripped} throughput region where
$\delta_5(T)$ is a decreasing function of $T$. 

\begin{remark}
A similar analysis shows that the \emph{2-Lookahead Myopic Policy}
achieves at least $94\%$ of $\mathbf{\Lambda}_s$, while the
\emph{3-Lookahead Myopic Policy} achieves at least $96\%$ of
$\mathbf{\Lambda}_s$. The $k$-Lookahead Myopic Policy is the same as
before except that the following weight functions are used for
scheduling decisions: Assuming the server is with queue 1 at time
slot $t$,\\
$W_1(t)
=Q_1(jT)\big(C_1(t)+\sum_{\tau=1}^k\mathbb{E}\big\{C_1(t+\tau)|C_1(t)\}\big)$
and $W_2(t) =Q_2(jT)\sum_{\tau=1}^k\mathbb{E}\big
\{C_2(t+\tau)|C_2(t)\big\}$.
\end{remark}

\section{Numerical Results}\label{Sec:Sim}
\begin{figure}
\centering
\includegraphics[width=0.4\textwidth]{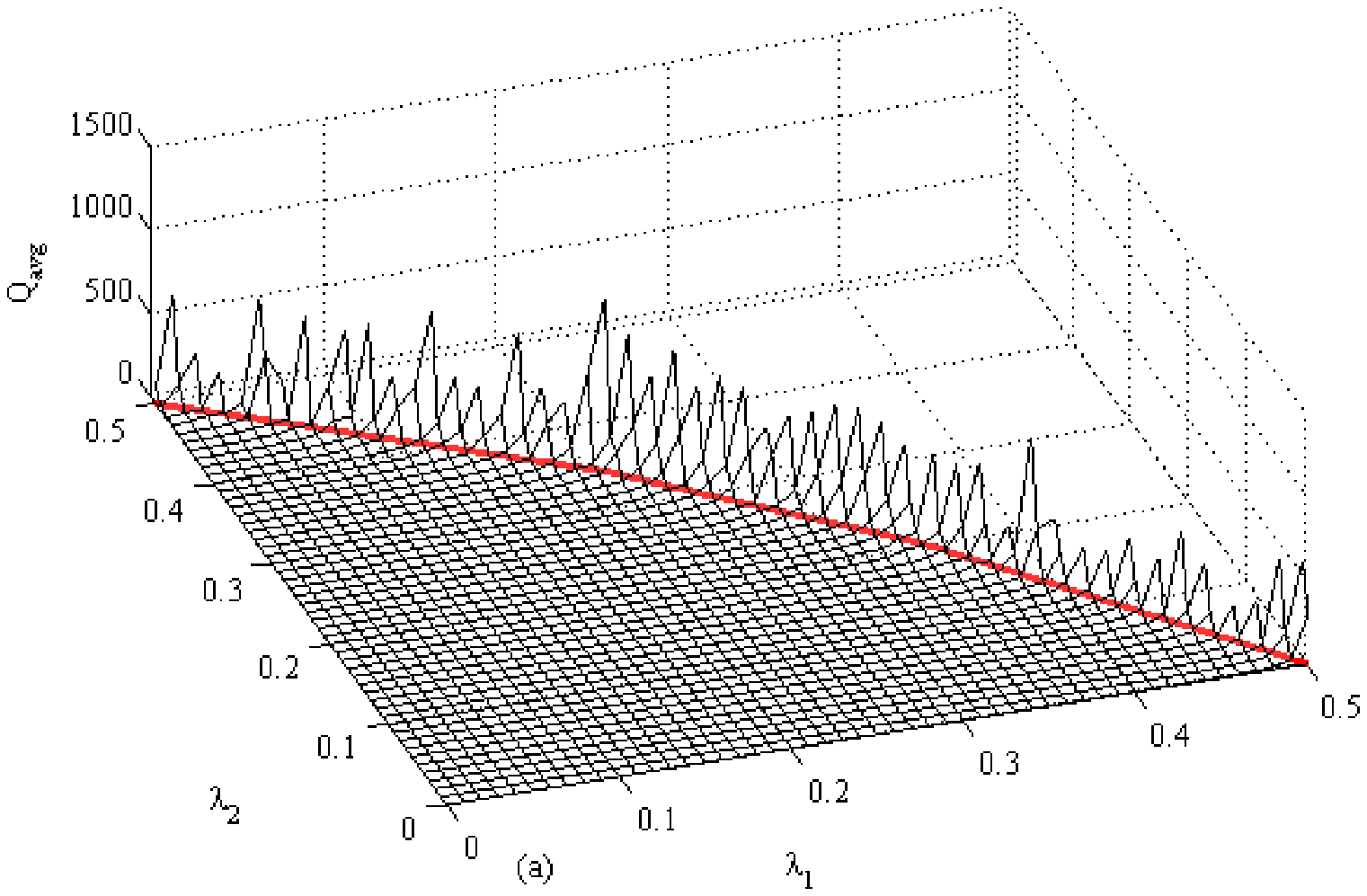}\\
\vspace{-0.5mm}
\includegraphics[width=0.4\textwidth]{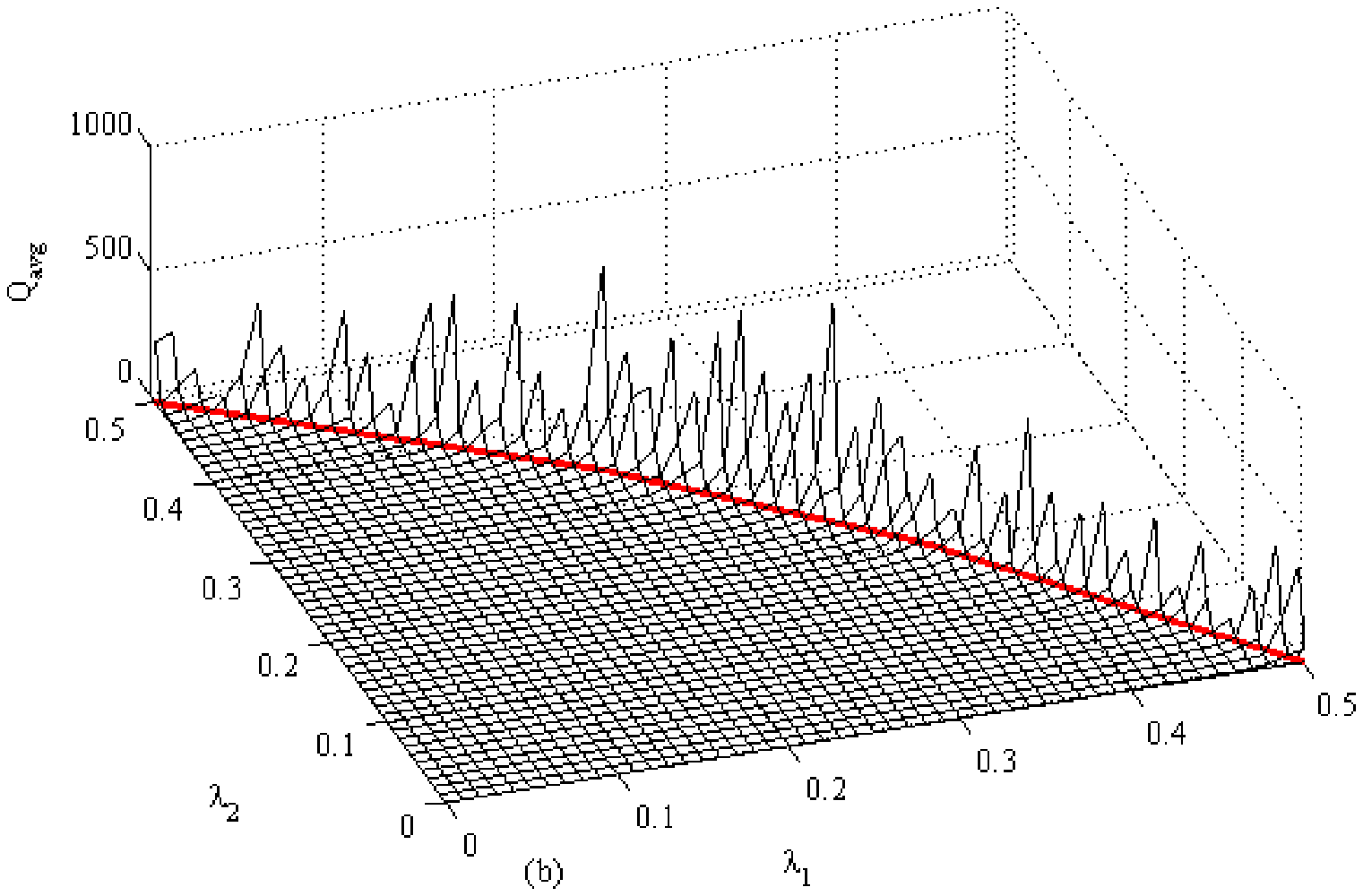}
\vspace{-2.5mm}\captionsetup{font=footnotesize}\caption{The total
average queue size for (a) the FBDC policy and (b) the Myopic Policy
for $T=10$ and $\epsilon = 0.40$.}\label{Fig:FBDCvsMyopicT10eps040}
\end{figure}
\begin{figure}
\centering
\includegraphics[width=0.4\textwidth]{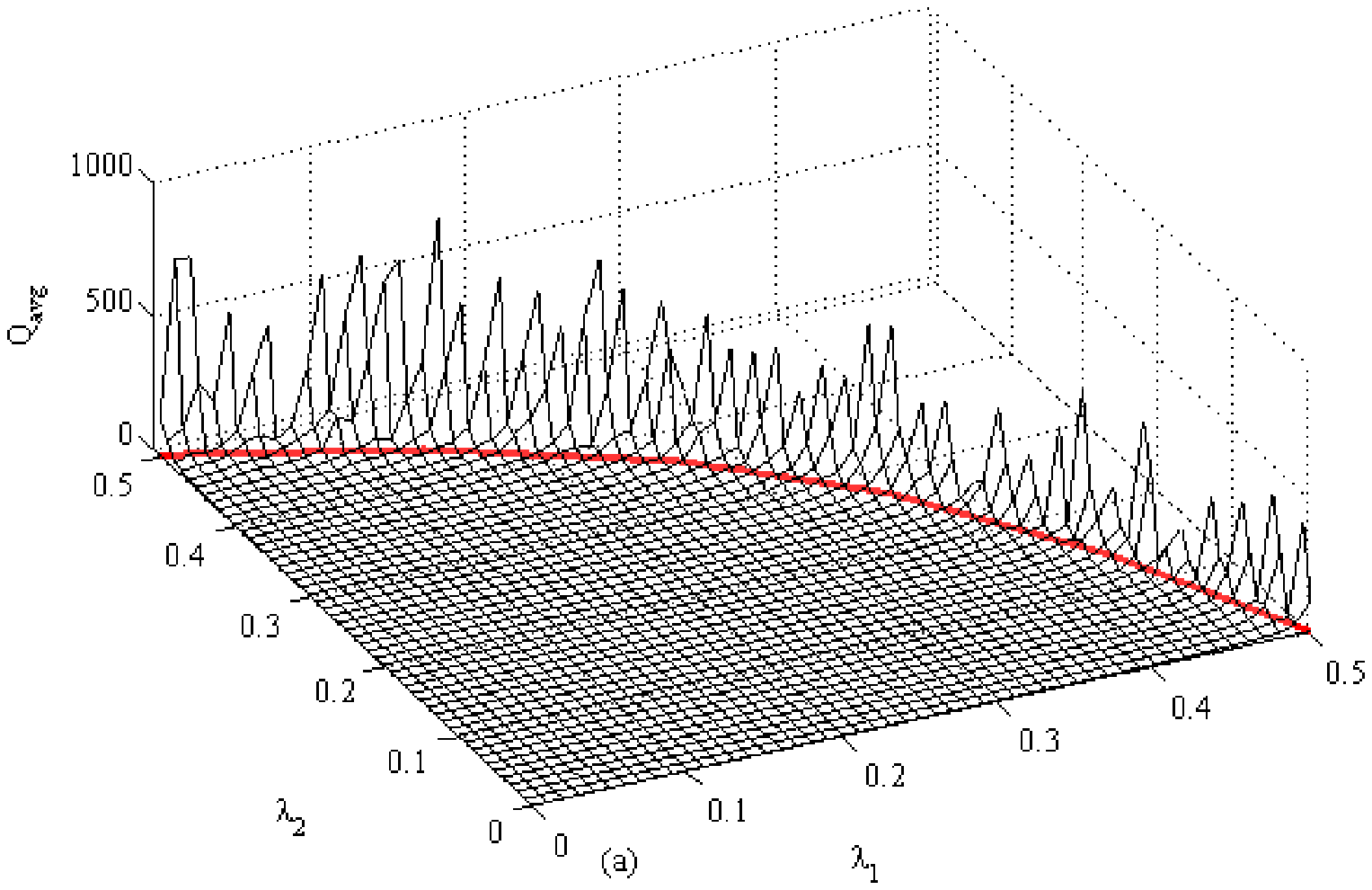}\\
\vspace{-0.5mm}
\includegraphics[width=0.4\textwidth]{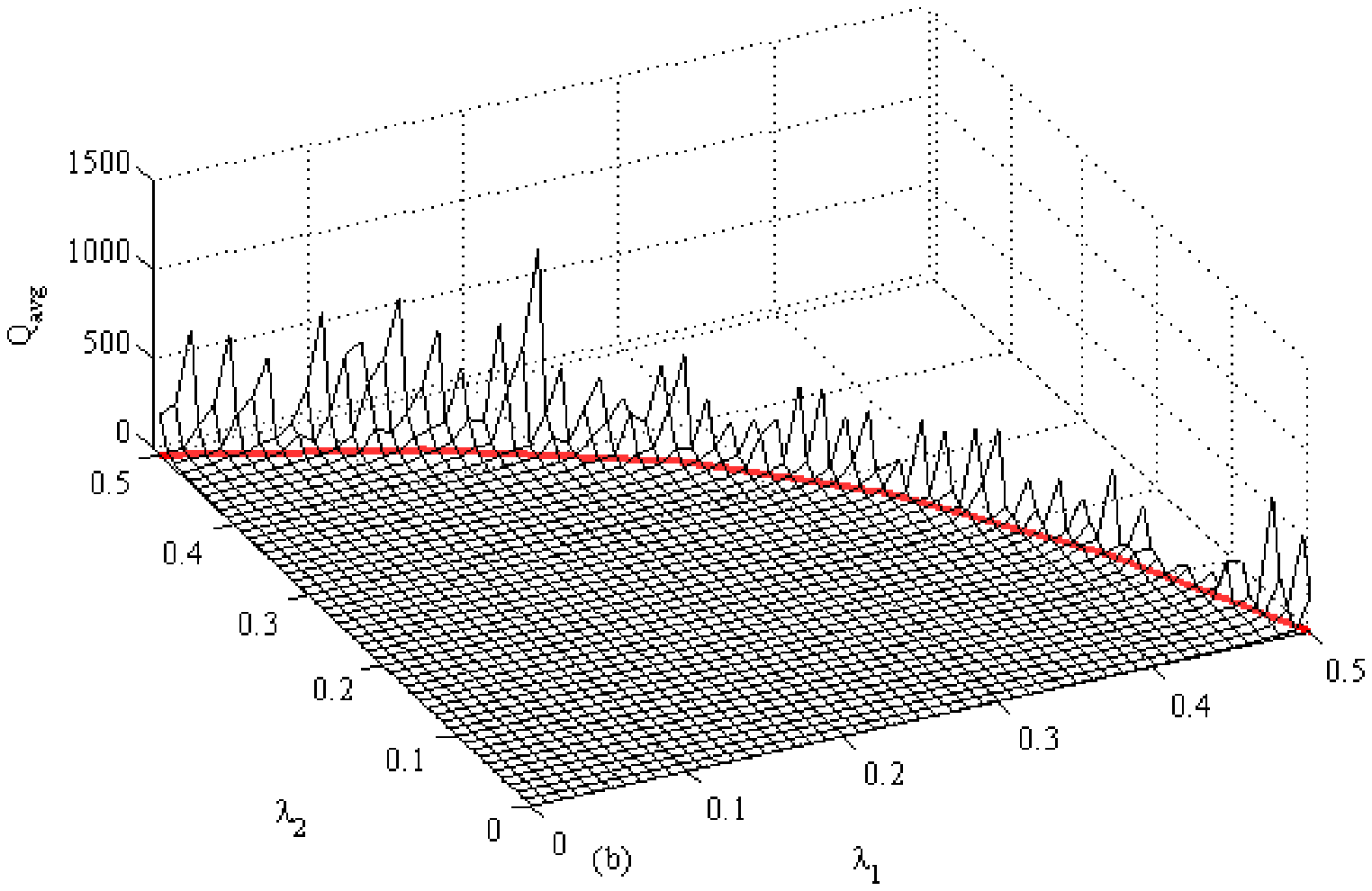}
\vspace{-2.5mm}\captionsetup{font=footnotesize}\caption{The total
average queue size for (a) the FBDC policy and (b) the Myopic policy
for $T=25$ and $\epsilon = 0.25$.}\label{Fig:FBDCvsMyopicT25eps025}
\end{figure}

In this section we present simulation results for the FBDC and the 1-Lookahead Myopic policies. We also
present numerical results that show the stability region for different $\epsilon$ values.
We performed simulation experiments that present average queue
occupancy results for the FBDC and the Myopic policies. We first
verified the correctness of the simulation model by confirming that
the FBDC policy achieves the full stability region in the simulation
results and then performed experiments for the 1-Lookahead Myopic (OLM) policy. 
In all the reported results, we have $(\lambda_1, \lambda_2) \in \mathbf{\Lambda}_s$ with $0.01$ increments.
For each point at the boundary of $\mathbf{\Lambda_s}$, we simulated one
point outside the stability region. Furthermore, for each
data point, the arrival processes were i.i.d., the channel processes
were Markovian as in Fig. \ref{Fig:channel_MC} and the simulation
length was
100,000 slots. 

Fig.\ \ref{Fig:FBDCvsMyopicT10eps040}~(a) presents the total average
queue size,
$\mathbf{Q}_{avg}\triangleq\sum_{t=1}^{100K}(Q_1(t)+Q_2(t))/t$,
under the FBDC policy for $\epsilon=0.40>\epsilon_c$. The boundary
of the stability region is shown by (red) lines on the two dimensional
$\lambda_1-\lambda_2$ plane. We observe that the average queue sizes
are small for all $(\lambda_1,\lambda_2) \in \mathbf{\Lambda}_s$ and
the big jumps in queue sizes occur for points outside
$\mathbf{\Lambda}_s$. Fig.\ \ref{Fig:FBDCvsMyopicT10eps040}~(b)
presents the performance of the OLM policy for the
same system. 
The
simulation results suggest that 
there is no appreciable difference between the stability regions of the
FBDC and the OLM policies. 
Note that the total average queue size is proportional to the
average delay in the system through Little's law. For these two
figures, the average delay under the OLM policy is less than that under the FBDC policy 
for
$86\%$ of all arrival rates considered. 
For the same system we also simulated a \emph{non-frame-based}
Myopic policy that utilizes the queue length information in the
\emph{current} time slot for the weight calculations in
(\ref{eq:myopic}). This implementation of the OLM policy preserves a
similar stability region to Fig.~\ref{Fig:FBDCvsMyopicT10eps040}~(b)
while having delay results at most as much as the FBDC policy for
$96\%$ of all arrival rates.
%
When the current queue lengths are used in the scheduling decisions,
the Myopic policy adapts to changes in the system dynamics more
quickly, therefore, better delay performance is expected.

Fig.\ \ref{Fig:FBDCvsMyopicT25eps025}~(a) shows the total average
queue size under (a) the FBDC policy and (b) the OLM policy for
$\epsilon=0.25<\epsilon_c$. Again this result suggest that the OLM
policy is achieving the full stability region. In this case the
regular and the non-frame-based implementations of the OLM policy
outperformed the FBDC policy in terms of delay for $81\%$ and $96\%$
of all arrival rates considered respectively.
%
%
%
%
These delay results show that the OLM policy is not only simpler to
implement than the FBDC policy, but it can also be more delay
efficient.

\begin{figure}
\centering
\includegraphics[width=0.4\textwidth]{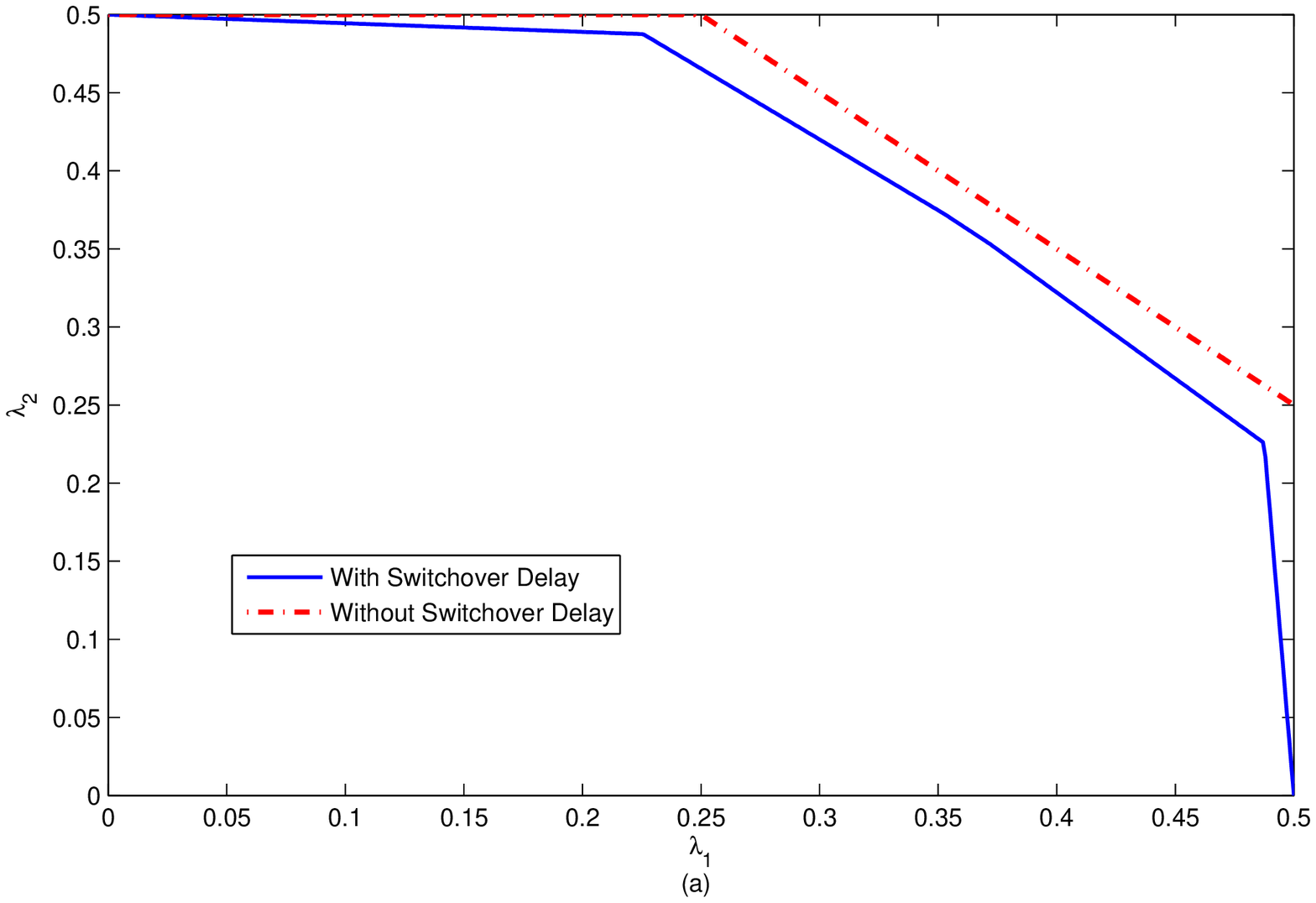}\\
\vspace{-0.5mm}
\includegraphics[width=0.4\textwidth]{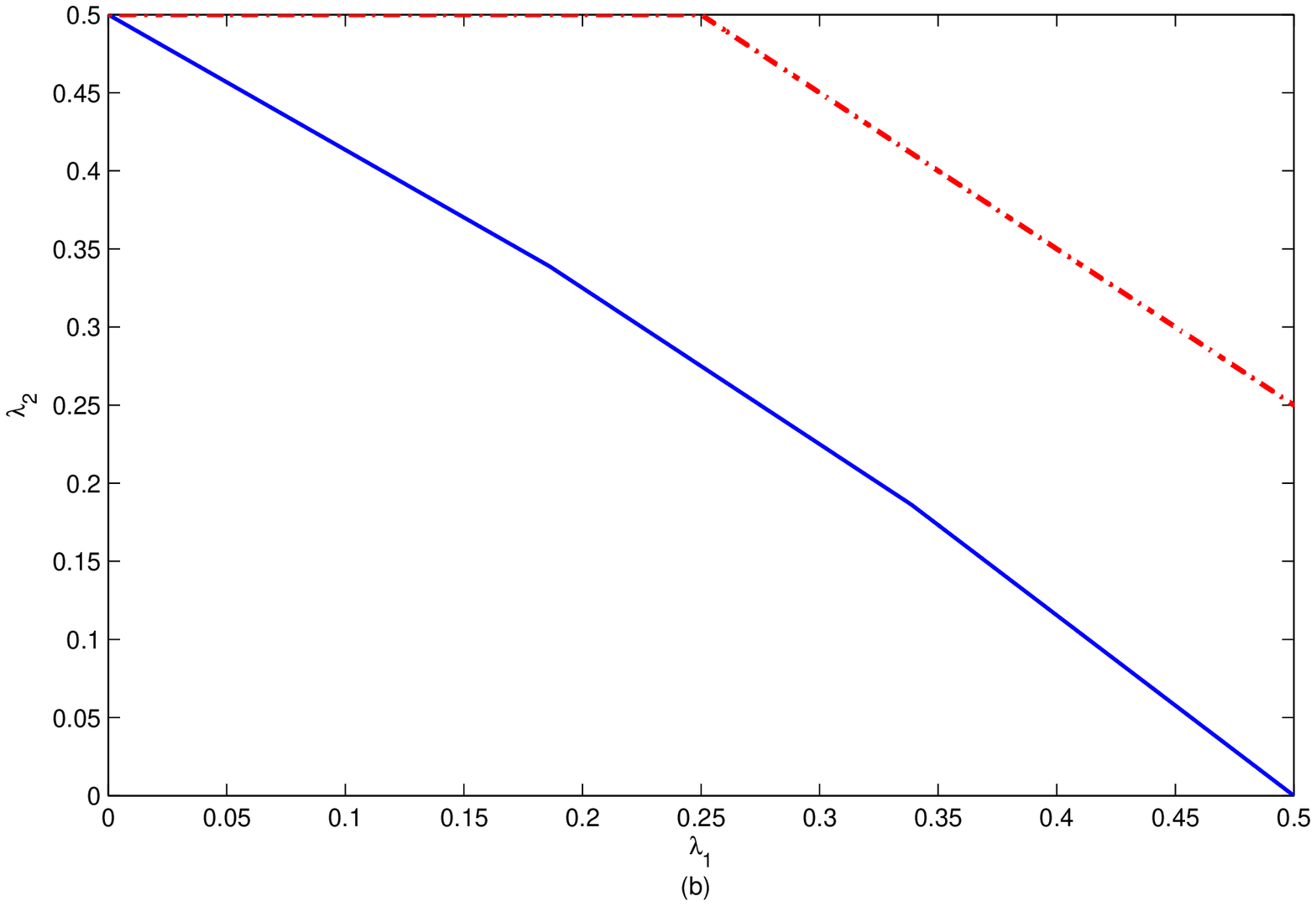}
\vspace{-2.5mm}\captionsetup{font=footnotesize}\caption{Stability
region under correlated channels with and without switchover time
for (a) $\epsilon=0.05$ and (b) $\epsilon =0.45$.}\label{Fig:cap_eps_0.05and0.45}
\end{figure}

Figures~\ref{Fig:cap_eps_0.05and0.45} (a) and (b) displays the stability region of the system
for $\epsilon = 0.05$ and $\epsilon = 0.45$ respectively. For very small channel correlation
($\epsilon \rightarrow 0.5$) the stability region tends to that of the i.i.d. channels case, whereas
for very large channel correlation ($\epsilon \rightarrow 0$) the stability region approaches that
of the no-switchover time case analyzed in \cite{tass93}.

\begin{figure}
\centering
\includegraphics[width=0.4\textwidth]{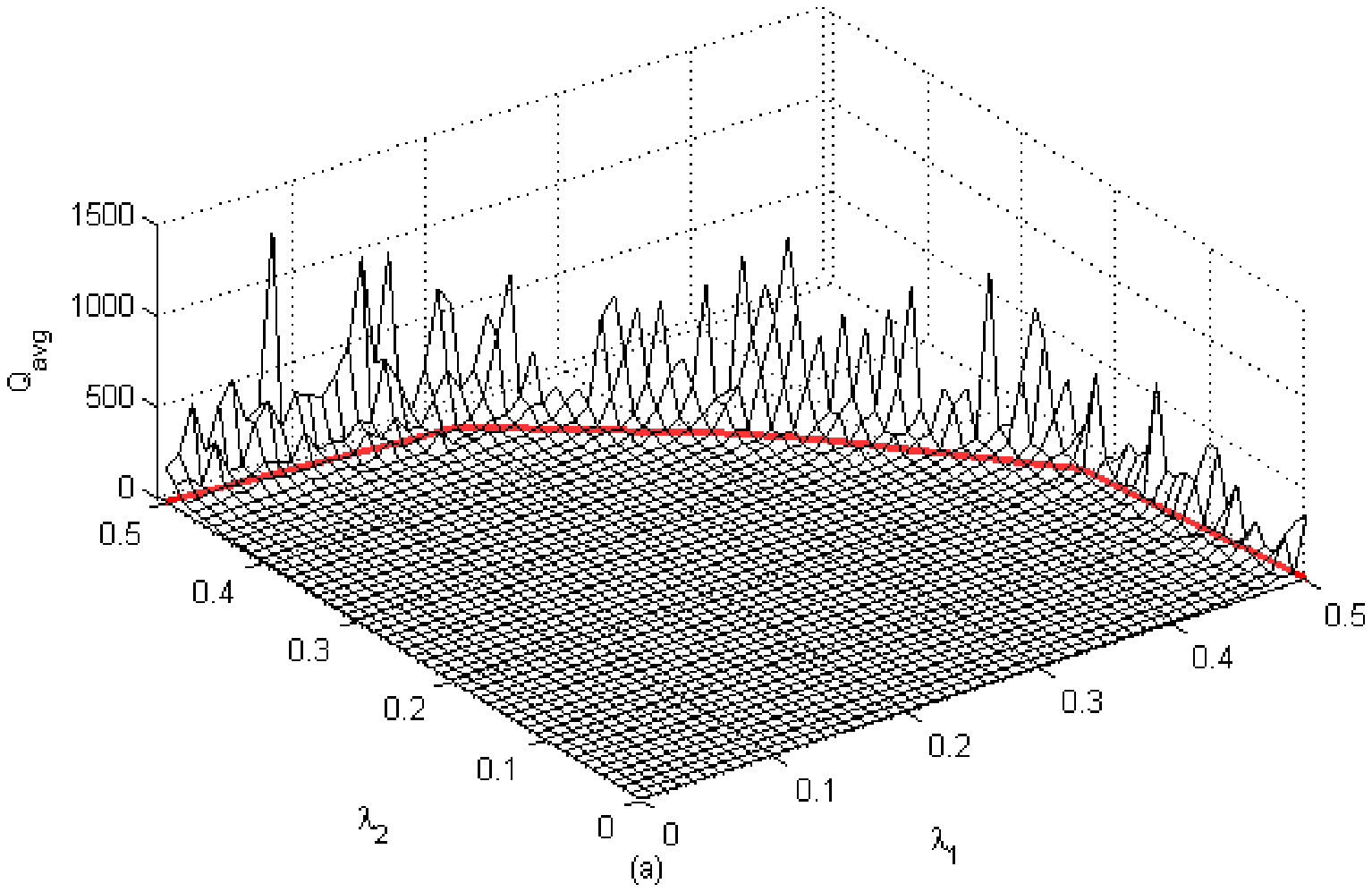}\\
\vspace{-0.5mm}
\includegraphics[width=0.4\textwidth]{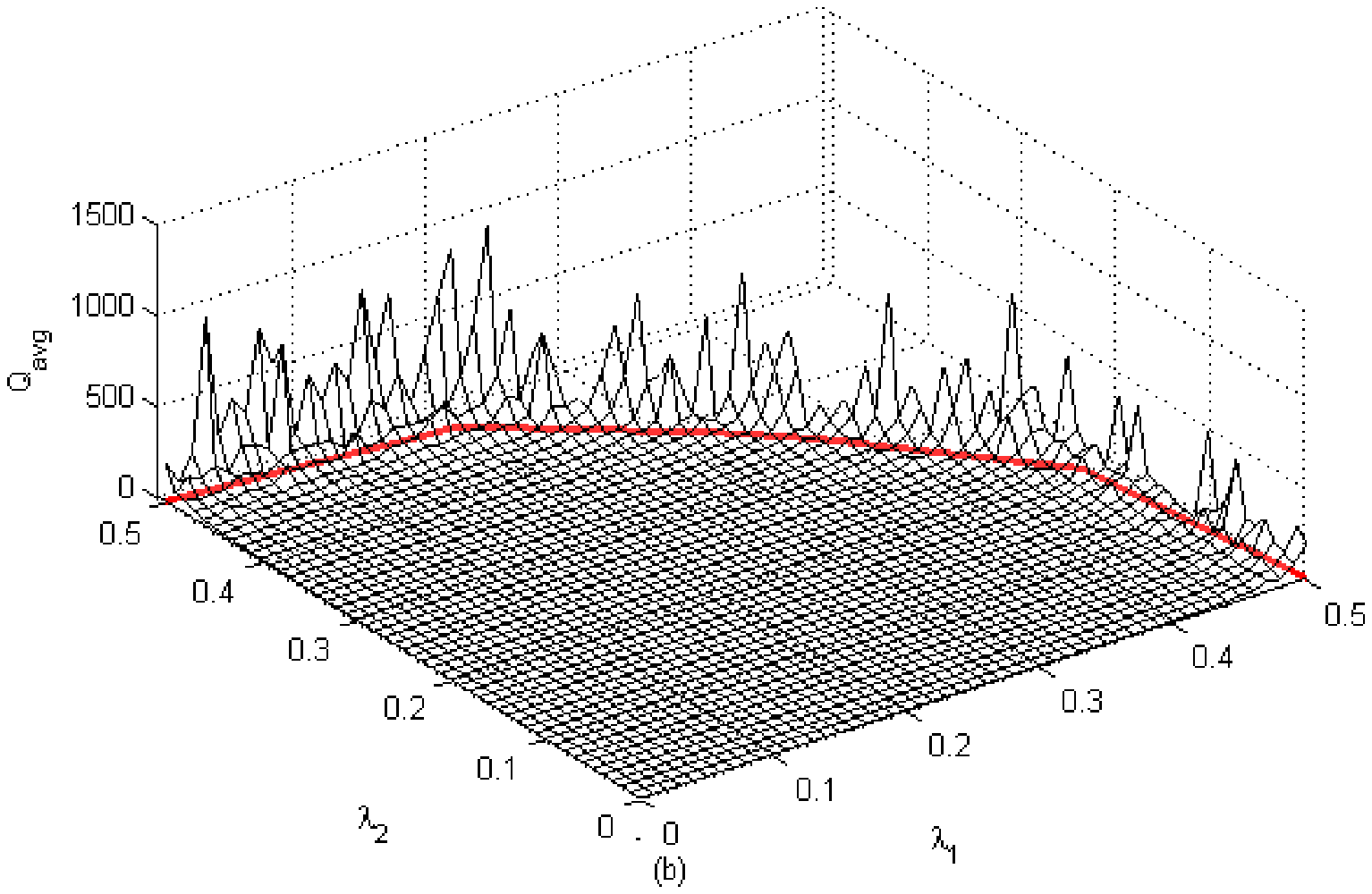}
\vspace{-2.5mm}\captionsetup{font=footnotesize}\caption{The total
average queue size for (a) the FBDC policy and (b) the Myopic policy
for $T=20$ and $\epsilon = 0.10$.}\label{Fig:FBDCvsMyopicT20eps010}
\end{figure}
Fig.\ \ref{Fig:FBDCvsMyopicT20eps010}~(a) shows the total average
queue size under (a) the FBDC policy and (b) the OLM policy for
$\epsilon=0.10 < \epsilon_c$. Again this result suggest that the OLM
policy is achieving the full stability region. In this case the
regular and the non-frame-based implementations of the OLM policy
outperformed the FBDC policy in terms of delay for $47\%$ and $91\%$
of all arrival rates considered respectively. This suggest that the delay advantage of
the Myopic policies are less pronounced for highly correlated channels.

\begin{figure}
\centering
\includegraphics[width=0.4\textwidth]{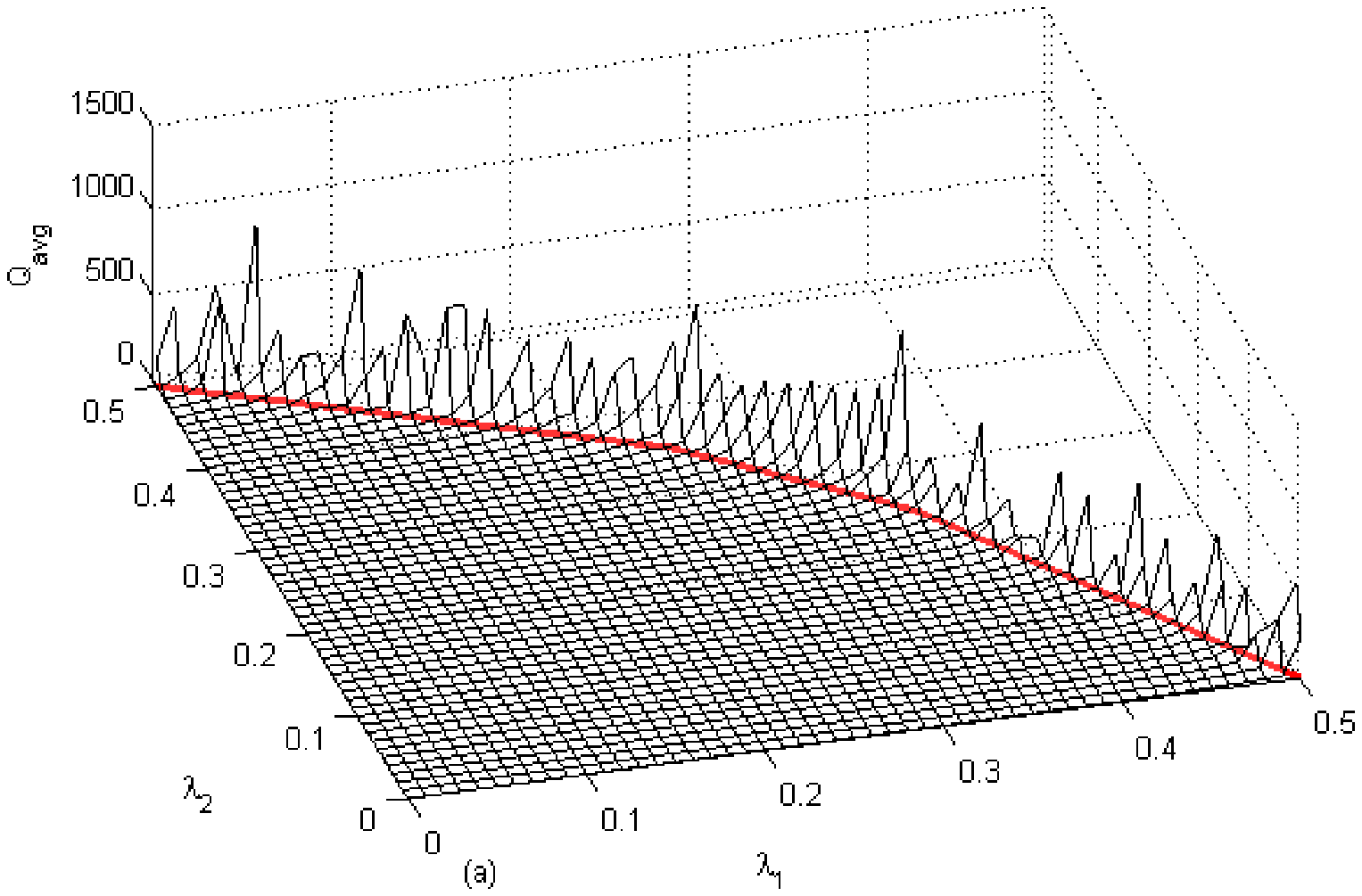}\\
\vspace{-0.5mm}
\includegraphics[width=0.4\textwidth]{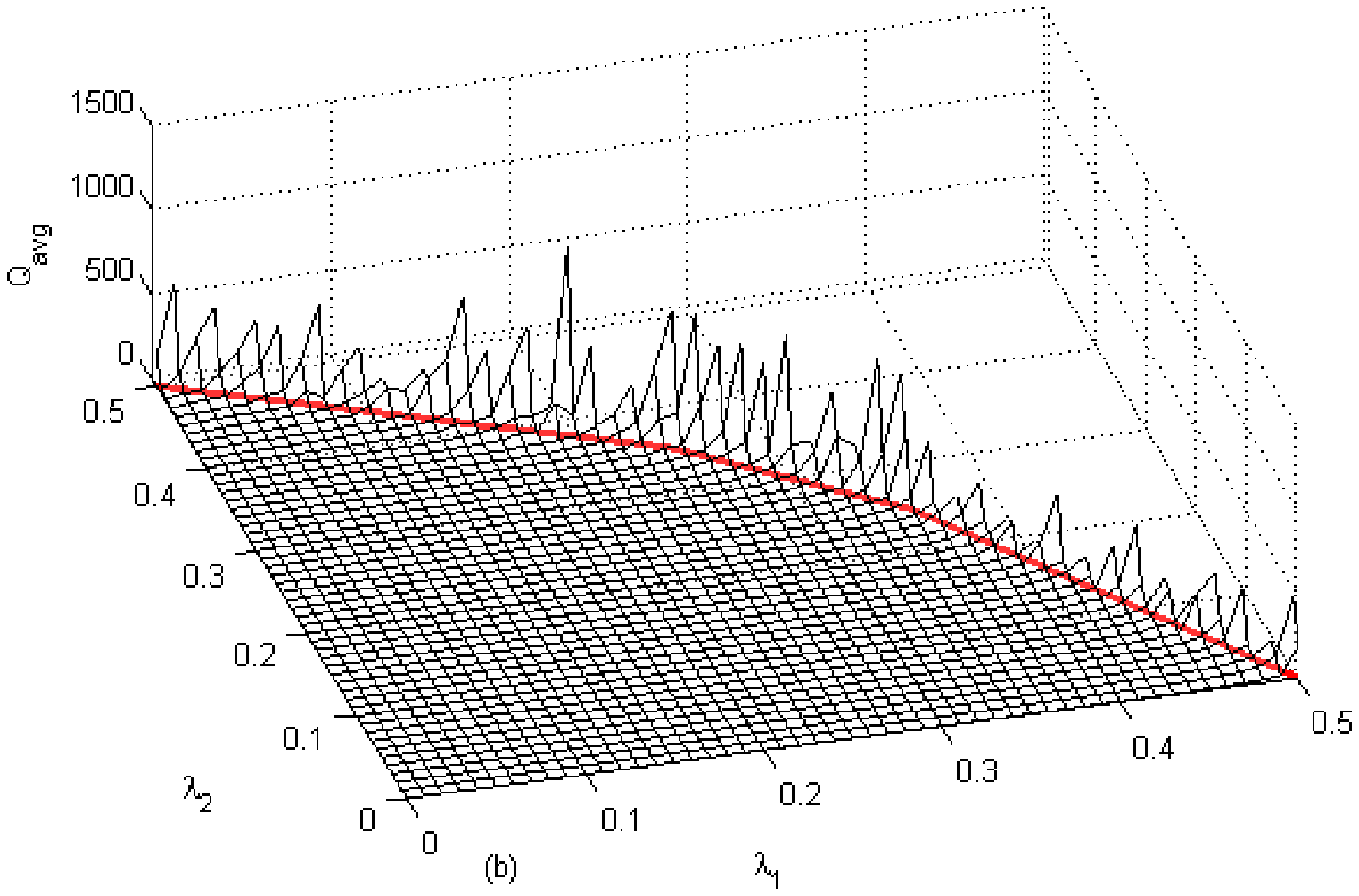}
\vspace{-2.5mm}\captionsetup{font=footnotesize}\caption{The total
average queue size for (a) the FBDC policy and (b) the Myopic policy
for $T=50$ and $\epsilon = 0.30$.}\label{Fig:FBDCvsMyopicT50eps030}
\end{figure}
Fig.\ \ref{Fig:FBDCvsMyopicT50eps030}~(a) shows the total average
queue size under (a) the FBDC policy and (b) the OLM policy for
$\epsilon=0.30 \simeq \epsilon_c$.  In this case the
regular and the non-frame-based implementations of the OLM policy
outperformed the FBDC policy in terms of delay for $94\%$ and $96\%$
of all arrival rates considered respectively. This suggest that the delay advantage of
the Myopic policies are more pronounced for less correlated channels.

\section{Conclusions}\label{Sec:Conc}

In this paper, we analyzed the scheduling problem with
\emph{randomly varying connectivity} and \emph{server switchover
time} for the first time in literature. We analytically
characterized the throughput region of the system using MDP theory,
developed a frame based dynamic control policy (FBDC) that is
throughput-optimal and developed much simpler \emph{Myopic Policies}
achieving $\gamma$-fraction of the throughput
region where $\gamma$ can be as high as $94\%$. 
For systems with correlated channels, throughput region
characterization in terms of the state action frequencies of the the
saturated system and the throughput-optimality of the FBDC policy
hold for general systems with many queues, arbitrary switching times
and more complicated Markovian channels. Similarly, the throughput
region as well as the throughput-optimality of the Gated policy for
the uncorrelated channels case hold for more general systems.

FBDC policy provides a new framework for developing
throughput-optimal policies for network control. 
For any queuing system whose corresponding saturated system is
finite-state Markovian, FBDC achieves stability based on a novel
idea of applying state action frequencies that solve an LP for the
saturated system.

In the future, we intend to derive analytical expressions for the
throughput regions of more general systems. In particular, for
systems with non-symmetric Markov channels or multiple-slot
switching times, analytical solution of the LP describing the
throughput region of the system could be possible. We intend to
develop throughput optimal Myopic policies for the current system
and for more general systems. Finally, scheduling and routing in
multihop wireless networks with dynamic channels and switchover
times is an interesting and challenging future direction.

\section*{Appendix A-Proof of Theorem \ref{thm:iid_ness} }\label{Sec:App-A}
%

We prove Theorem \ref{thm:iid_ness} for a more general system with
$N$-queues and travel time between queue-$i$ and queue-$j$ given by
$D_{ij}$ slots. We call the term $\sum_{i=1}^N \lambda_i/p_i$ the
system load and denote it by $\rho$ since it is the rate with which
the work is entering the system in the form of service slots.
We prove that a necessary condition for the stability of any policy is $\rho
=\sum_{i=1}^N \lambda_i/p_i<1$.
\begin{proof}
Since queues have memoryless channels, for any received packet, as
soon as the server switches to queue $i$, the expected time to ON
state is $1/p_i$. Namely, the time to ON state is a geometric random
variable with parameter $p_i$ and hence $1/p_i$ is essentially the
``service time per packet'' for queue-$i$. 
%
%
Therefore, the i.i.d. connectivity is essentially a geometric random
variable representing service time in a classical polling system. In
a multiuser single-server system \emph{with or without switchover
times}, with stationary arrivals whose average arrival rates are
$\lambda_i, i \in \{1,2\}$, and i.i.d. service times independent of
arrivals with average service times $1/p_i, i \in \{1,2\}$, a
necessary condition for stability is given by the system load,
$\rho$, less than 1. To see this, consider the polling system with
zero switchover times, stationary arrivals of rate $\lambda_i$ and
$i.i.d.$ service times of mean $1/p_i$. The throughput region of
this system is an upperbound on the throughput region of the
corresponding system with nonzero switchover times (for the same
sample path of arrival and channel processes, the system with zero
switchover time can achieve exactly the same departure process as
the system with nonzero switchover times by making the server idle
when necessary). A necessary condition for the stability of the
former system is $\rho = \lambda_1/p_1 +
\lambda_N/p_N+...+\lambda_1/p_N <1$, (e.g., \cite{Walrand}).
\end{proof}

\section*{Appendix B-Proof of Theorem \ref{thm:iid_suff} }

Again we prove the theorem for a more general system with multiple
queues and travel time between queue-$i$ and queue-$j$ given by
$D_{ij}$ slots. We prove that
Gated cyclic policy is stable if $\rho
=\sum_{i=1}^N \frac{\lambda_i}{p_i}$.
\begin{proof}
Let $m$ be the discrete time index for the $m$th time the mobile
stops for servicing a queue. The proof is similar to the stability
proof in \cite{AltKonsLiu92}. Let $T_m$ be the time slot number of
this mobile-node meeting times (at time $T_{m+1}$ the mobile meets
with the next node in the cycle and at time $T_{m+N}$ it comes back
to the same node). Let $I(m)$ be the i.d. of the node that the
mobile serves at time $T_m$ and let $S(Q_{I(m)}(T_m))$ be the
service time required to serve $Q_{I(m)}(T_m)$ packets at time
$T_{m}$. Since we have cyclic service, we specify one particular
order of service and simplify the notation for traveling times from
node $i$ to node $j$, $D_{ij}$ as $D_i$ denoting the time required
to move from node $i$ to the next node in the cycle. Also let
$D=\sum_{i=1}^N D_i$ be the total travel time in one cycle.

Since we have gated service, we obtain the following queue
evolution:
\begin{equation}\label{eq:queue evol}
\displaystyle \sum_{i=1}^N Q_{i}(T_{m+1}) = \sum_{i=1}^N Q_{i}(T_m)
+ \sum_{i=1}^N \Big( \sum_{t=T_m}^{T_m+S(Q_{I(m)}(T_m))+D_{I(m)}-1}
A_{i}(t) \Big) - Q_{I(m)}(T_m).
\end{equation}
Consider the following Lyapunov function:
\begin{equation}
\displaystyle L(\vec{Q}(T_{m})) = \sum_{i=1}^N
\frac{Q_i(T_{m})}{p_i}.
\end{equation}
The intuition behind this choice of Lyapunov function is that, given
the current queue sizes, it is the expected amount of service time
needed to serve what is currently in all the queues. From
(\ref{eq:queue evol}) we obtain,
\begin{equation}\label{eq:queue evol2}
\displaystyle \sum_{i=1}^N \frac{Q_{i}(T_{m+1})}{p_i} = \sum_{i=1}^N
\frac{Q_{i}(T_m)}{p_i} + \sum_{i=1}^N \Big(
\sum_{t=T_m}^{T_m+S(Q_{I(m)}(T_m))+D_{I(m)}-1} \frac{A_{i}(t)}{p_i}
\Big) - \frac{Q_{I(m)}(T_m)}{p_{I(m)}}.
\end{equation}
Taking expectations conditional on $\mathbf{Q}(T_m)$ we obtain,
\begin{equation}\label{eq:queue evol3}
\displaystyle \mathbb{E}\Big[\sum_{i=1}^N
\frac{Q_{i}(T_{m+1})}{p_i}\Big|\mathbf{Q}(T_m)\Big] = \sum_{i=1}^N
\frac{Q_{i}(T_m)}{p_i} + \sum_{i=1}^N \frac{\lambda_{i}}{p_i}
\Big(\frac{Q_{I(m)}(T_m)}{p_{I(m)}}+D_{I(m)} \Big) -
\frac{Q_{I(m)}(T_m)}{p_{I(m)}}.
\end{equation}
where we used the independence of the arrival and the channel
processes conditional on the current queue sizes. $S(Q_{I(m)}(T_m))$
is a random variable that depends on arrivals before $T_m$ but not
on arrivals after $T_m$ as the arrival processes are i.i.d. over
time. Therefore, $\textrm{E}[S(Q_{I(m)}(T_m))|Q_{I(m)}(T_m)]$ is
nothing but $\frac{Q_{I(m)}(T_m)}{p_{I(m)}} $. Simplifying we obtain
\begin{equation}\label{eq:queue evol4}
\displaystyle \mathbb{E}\Big[\sum_{i=1}^N
\frac{Q_{i}(T_{m+1})}{p_i}\Big|\mathbf{Q}(T_m)\Big] = \sum_{i=1}^N
\frac{Q_{i}(T_m)}{p_i} +  \rho D_{I(m)} -
\frac{Q_{I(m)}(T_m)}{p_{I(m)}}(1-\rho).
\end{equation}
Now we write a similar expression for time $T_{m+2}$.
\begin{eqnarray}
\displaystyle \mathbb{E}\Big[\sum_{i=1}^N
\frac{Q_{i}(T_{m+2})}{p_i}\Big|\mathbf{Q}(T_{m})\Big]
\!\!\!\!\!\!\!\!\!&&= \mathbb{E}\bigg\{ \mathbb{E}\Big[ \sum_{i=1}^N
\frac{Q_{i}(T_{m+2})}{p_i}|\mathbf{Q}(T_{m+1})\Big] \Big|
\mathbf{Q}(T_m) \bigg\}\nonumber\\
&&= \mathbb{E}\bigg\{\sum_{i=1}^N \frac{Q_{i}(T_{m+1})}{p_i} + \rho
D_{I(m+1)} - \frac{Q_{I(m+1)}(T_{m+1})}{p_{I(m+1)}}(1-\rho)\Big|
\mathbf{Q}(T_m) \bigg\}. \label{eq:queue evol5}
\end{eqnarray}
Noting that $Q_{I(m+1)}(T_{m+1}) \ge Q_{I(m+1)}(T_{m})$ and using
(\ref{eq:queue evol4}), we have from (\ref{eq:queue evol5})
\begin{equation*}
\displaystyle \mathbb{E}\Big[\sum_{i=1}^N
\frac{Q_{i}(T_{m+2})}{p_i}\Big|\mathbf{Q}(T_{m})\Big] \le
\sum_{i=1}^N \frac{Q_{i}(T_{m})}{p_i} + \rho (D_{I(m)}+D_{I(m+1)}) -
(1-\rho) \Big ( \frac{Q_{I(m)}(T_{m})}{p_{I(m)}}+
\frac{Q_{I(m+1)}(T_{m})}{p_{I(m+1)}} \Big).
\end{equation*}
Repeating the same argument we obtain a drift condition over one
cycle given by
\begin{equation}\label{eq:queue evol6}
\displaystyle \mathbb{E}\Big[\sum_{i=1}^N
\frac{Q_{i}(T_{m+N})}{p_i}-\frac{Q_{i}(T_{m})}{p_i}\Big|\mathbf{Q}(T_{m})\Big]
\le \rho D - (1-\rho) \sum_{i=0}^{N-1}
\frac{Q_{I(m+j)}(T_{m})}{p_{I(m+j)}}.
\end{equation}
Hence, we obtain a negative drift as soon as 
\begin{equation}\label{eq:temp7}
\displaystyle \sum_{i=1}^N \frac{Q_{i}(T_m)}{p_i} >
\rho\frac{D}{1-\rho}.
\end{equation}
Therefore using the Lyapunov stability (e.g., \cite[Theorem 3]{neely05}), the queue length
processes at discrete times indexed by $m$ satisfies an $N$-step
negative Lyapunov drift and therefore they are stable. Now consider
an arbitrary time slot $t\in \ (T_m,T_{m+1})$. We have that $Q(t)
\le Q(T_{m+1})$ since there is guaranteed to be no service between
$T_m$ and $T_{m+1}$. Therefore we have $\mathbb{E}\{ Q(t)\} \le
\mathbb{E}\{ Q(T_{m+1})\}$. Therefore, the system is stable as long
as $\rho <1$.
\end{proof}

\section*{Appendix C-Proof of Theorem \ref{thm:stab}}\label{Sec:App-C}

We enumerate the states as follows:
\begin{equation}\label{eq:state_enum}
\begin{array}{llll}
s=(1,1,1)\equiv1,\;\;\; &s=(1,1,0)\equiv2,\;\;\;
&s=(1,0,1)\equiv3,\;\;\;  & s=(1,0,0)\equiv4,\;\;\;\\
s=(2,1,1) \equiv5 & s=(2,1,0)\equiv6 & s=(2,0,1) \equiv7& s=(2,0,0) \equiv8.\\
\end{array}
\end{equation}
We rewrite the balance equations in (\ref{eq:LP}) in more details.
\beqn x(1;1)+x(1;0) = (1-\epsilon)^2 \big(x(1;1)+x(5;0)\big)
\!\!\!\!\!&+&
\!\!\!\!\!\epsilon(1-\epsilon)\big(x(2;1)+ x(6;0)\big)\nonumber\\
+ \epsilon(1-\epsilon)\big(x(3;1)+x(7;0)\big) \!\!\!\!\!&+&\!\!\!\!\! \epsilon^2\big(x(4;1)+x(8;0)\big)\label{eq:saf1}  \\
x(2;1)+x(2;0) = \epsilon(1-\epsilon) \big(x(1;1)+x(5;0)\big)
\!\!\!\!\!&+&\!\!\!\!\!
(1-\epsilon)^2\big(x(2;1) + x(6;0)\big) \nonumber\\
\!\!\!\!\!+ \epsilon^2\big(x(3;1)+x(7;0)\big)
\!\!\!\!\!&+&\!\!\!\!\!
\epsilon(1-\epsilon)\big(x(4;1)+x(8;0)\big)\label{eq:saf2}\\
\!\!\!\!\!&\ldots& \nonumber\\
x(5;1)+x(5;0) = (1-\epsilon)^2 \big(x(5;1)+x(1;0)\big)
\!\!\!\!\!\!\!\!&+\!\!\!\!\!&
\epsilon(1-\epsilon)\big(x(6;1)+ x(2;0)\big)\nonumber\\
+ \epsilon(1-\epsilon)\big(x(7;1)+x(3;0)\big) \!\!\!\!\!&+&\!\!\!\!\! \epsilon^2\big(x(8;1)+x(4;0)\big) \label{eq:saf3} \\
x(7;1)+x(7;0)= \epsilon(1-\epsilon) \big(x(5;1)+x(1;0)\big)
\!\!\!\!\!&+&\!\!\!\!\!
\epsilon^2\big(x(6;1) + x(2;0)\big) \nonumber\\
+ (1-\epsilon)^2 \big(x(7;1)+x(3;0)\big) \!\!\!\!\!&+&\!\!\!\!\!
\epsilon(1-\epsilon)\big(x(8;1)+x(4;0)\big)\label{eq:saf4}\\
&\ldots\nonumber  \eeqn\\
The following equations hold for each channel state pair
$(C_1,C_2)$. \beqn
x(1;1)+x(1;0)+ x(5;1)+x(5;0) = 1/4 \label{eq:state_1}\\
x(2;1)+x(2;0)+ x(6;1)+x(6;0) = 1/4 \label{eq:state_2}\\
x(3;1)+x(3;0)+ x(7;1)+x(7;0) = 1/4 \label{eq:state_3}\\
x(4;1)+x(4;0)+ x(8;1)+x(8;0) = 1/4 \label{eq:state_4}
 \eeqn\\
 \\
Let $u_1=(x(1;1)+x(2;1)$ and $u_2=(x(5;1)+x(7;1))$. Summing up
(\ref{eq:saf1}) with (\ref{eq:saf2}) and (\ref{eq:saf3}) with
(\ref{eq:saf4}) we have
\begin{eqnarray*}
\epsilon u_1 \!\!\!\!&=&\!\!\!\! - \big(x(1;0)+x(2;0)\big) +
\epsilon\big( x(3;1)+x(4;1) \big) + \epsilon\big(x(7;0)+ x(8;0)
\big) + (1-\epsilon) \big( x(5;0)+x(6;0) \big)\\
\epsilon u_2 \!\!\!\!&=&\!\!\!\! - \big(x(5;0)+x(7;0)\big) +
\epsilon\big( x(6;1)+x(8;1) \big) + \epsilon\big(x(2;0)+ x(4;0)
\big) + (1-\epsilon) \big( x(1;0)+x(3;0) \big)\\
\end{eqnarray*}
\\
Rearranging and using (\ref{eq:state_1})-(\ref{eq:state_4}) we have
\begin{eqnarray}
u_1 \!\!\!\!&=&\!\!\!\! \frac{1-\epsilon}{2}+\epsilon\big(
x(3;1)+x(4;1)+x(7;0)+ x(8;0)\big) - (2-\epsilon) \big( x(1;0)+x(2;0) \big) - (1-\epsilon) \big( x(5;1)+x(6;1)\big)\nonumber\\
\label{eq:saf_work1}\\
u_2 \!\!\!\!&=&\!\!\!\! \frac{2-\epsilon}{4}+\epsilon\big(
x(2;0)-x(4;1) + x(6;1)-x(8;0) \big) - (2-\epsilon) \big(
x(5;0)+x(7;0) \big) - (1-\epsilon) \big(
x(1;1)+x(3;1)\big)\nonumber\\
\label{eq:saf_work2}
\end{eqnarray}
\\
\\
Using (\ref{eq:saf3}) in (\ref{eq:saf_work1})  and (\ref{eq:saf1})
in (\ref{eq:saf_work2}) we have
\begin{eqnarray}
u_1 \!\!\!\!&=&\!\!\!\! \frac{1-\epsilon}{2}+\epsilon\big(
x(3;1)+x(4;1)+x(7;0)+ x(8;0)\big) -
\frac{\epsilon(1-\epsilon)}{2-\epsilon} \big( x(4;0)+x(8;1) \big)
-\frac{(1-\epsilon)(3-2\epsilon)}{2-\epsilon} x(6;1)\nonumber \\
&+& \frac{1-\epsilon}{\epsilon(2-\epsilon)} x(5;0)
-\frac{1+\epsilon-\epsilon^2}{\epsilon(2-\epsilon)} x(1;0) -
\frac{(1-\epsilon)^2}{2-\epsilon} \big( x(3;0)+x(7;1) \big)
-\Big( 2-\epsilon + \frac{(1-\epsilon)^2}{2-\epsilon}\Big) x(2;0)\label{eq:saf_work3} \\
u_2 \!\!\!\!&=&\!\!\!\! \frac{2-\epsilon}{4}+\epsilon\big( x(2;0)+
x(6;1)\big) - \Big(
\epsilon+\frac{\epsilon(1-\epsilon)}{2-\epsilon}\Big) \big(
x(4;1)+x(8;0) \big)
-\frac{(1-\epsilon)(3-2\epsilon)}{2-\epsilon} x(3;1)\nonumber \\
&+& \frac{1-\epsilon}{\epsilon(2-\epsilon)} x(1;0)
-\frac{1+\epsilon-\epsilon^2}{\epsilon(2-\epsilon)} x(5;0) -
\frac{(1-\epsilon)^2}{2-\epsilon} \big( x(2;1)+x(6;0) \big) -\Big(
2-\epsilon + \frac{(1-\epsilon)^2}{2-\epsilon}\Big)
x(7;0).\label{eq:saf_work4}
\end{eqnarray}
\\
\\
Using (\ref{eq:state_3}) and (\ref{eq:state_4}) in
(\ref{eq:saf_work3}) and (\ref{eq:state_2}) in (\ref{eq:saf_work4})
we have
\begin{eqnarray}
u_1 \!\!\!\!&=&\!\!\!\!
\frac{(1-\epsilon)(3-2\epsilon)}{4(2-\epsilon)}+ \Big(\epsilon+
\frac{\epsilon(1-\epsilon)}{2-\epsilon} \Big) \big(
x(4;1)+x(8;0)\big) +  \frac{1}{2-\epsilon} \big(x(3;1)+x(7;0)\big)\nonumber\\
&-&\frac{(1-\epsilon)(3-2\epsilon)}{2-\epsilon} x(6;1)  +
\frac{1-\epsilon}{\epsilon(2-\epsilon)} x(5;0)
-\frac{1+\epsilon-\epsilon^2}{\epsilon(2-\epsilon)} x(1;0)
-\Big( 2-\epsilon + \frac{(1-\epsilon)^2}{2-\epsilon}\Big) x(2;0)\label{eq:saf_work5} \\
u_2 \!\!\!\!&=&\!\!\!\! \frac{3-2\epsilon}{4(2-\epsilon)}
-\Big(\epsilon+ \frac{\epsilon(1-\epsilon)}{2-\epsilon} \Big) \big(
x(4;1)+x(8;0)\big) + \frac{1}{2-\epsilon} \big(x(2;0)+x(6;1)\big)\nonumber\\
&-&\frac{(1-\epsilon)(3-2\epsilon)}{2-\epsilon} x(3;1) +
\frac{1-\epsilon}{\epsilon(2-\epsilon)} x(1;0) -
\frac{1+\epsilon-\epsilon^2}{\epsilon(2-\epsilon)} x(5;0) - \Big(
2-\epsilon + \frac{(1-\epsilon)^2}{2-\epsilon}\Big)
x(7;0).\label{eq:saf_work6}
\end{eqnarray}
\\

Consider the LP objective function $\alpha_1 \big(x(1;1)+x(2;1)\big)
+ \alpha_2\big(x(5;1)+x(7;1)\big)$, and note that the solution to
this LP is a stationary deterministic policy for any given
$\alpha_1$ and $\alpha_2$. This means that, for any state $s$ either
$x(s;1)$ or $x(s;0)$ has to be zero. In order to maximize $\alpha_1
\big(x(1;1)+x(2;1)\big) + \alpha_2\big(x(5;1)+x(7;1)\big)$ we need

\begin{eqnarray*}
x(7;0) \!\!\!\!&=&\!\!\!\! 0 \;\;\;\;\textrm{if  }\;\; \frac{\alpha_2}{\alpha_1} \ge \frac{1}{(2-\epsilon)^2+(1-\epsilon)^2},\\
x(3;1) \!\!\!\!&=&\!\!\!\! 0 \;\;\;\;\textrm{if  }\;\;\frac{\alpha_2}{\alpha_1} \ge \frac{1}{(1-\epsilon)(3-2\epsilon)},\\
x(5;0) \!\!\!\!&=&\!\!\!\! 0 \;\;\;\;\textrm{if  }\;\;\frac{\alpha_2}{\alpha_1} \ge \frac{1-\epsilon}{1+\epsilon-\epsilon^2},\\
x(8;0)=x(4;1)\!\!\!\!&=&\!\!\!\! 0 \;\;\;\;\textrm{if  }\;\;\frac{\alpha_2}{\alpha_1} \ge 1,\\
x(6;0) \!\!\!\!&=&\!\!\!\! 0 \;\;\;\;\textrm{if  }\;\;\frac{\alpha_2}{\alpha_1} \ge (1-\epsilon)(3-2\epsilon),\\
x(1;1) \!\!\!\!&=&\!\!\!\! 0
\;\;\;\;\textrm{if}\;\;\frac{\alpha_2}{\alpha_1} \ge
\frac{1+\epsilon-\epsilon^2}{1-\epsilon}.
\end{eqnarray*}
\\
Note that we have
\begin{eqnarray*}
(2-\epsilon)^2+(1-\epsilon)^2 &\ge& (1-\epsilon)(3-2\epsilon) \ge 1\\
(2-\epsilon)^2+(1-\epsilon)^2 &\ge& \frac{1+\epsilon-\epsilon^2}{1-\epsilon} \ge 1\\
\end{eqnarray*}
holding for all $\epsilon \in [0,0.5]$. Consider the following two
cases:
\\
\\
\textbf{Case-1: $\epsilon > \epsilon_c = 1-\sqrt{2}/2$}
\\
In this case we have $(1-\epsilon)(3-2\epsilon) <
(1+\epsilon-\epsilon^2)/(1-\epsilon)$. This means that we have the
following optimal
policies depending on the value of $\alpha_2/\alpha_1$.\\
\\
$1 \le \frac{\alpha_2}{\alpha_1} \le (1-\epsilon)(3-2\epsilon)$:
\begin{equation*}
\begin{array}{llll}
\textrm{@queue-1}: (1,1,1):\textrm{stay}, &(1,1,0):\textrm{stay}, &(1,0,1):\textrm{switch}, &(1,0,0):\textrm{switch}. \\
\textrm{@queue-2}: (2,1,1):\textrm{stay}, &(2,1,0):\textrm{switch},
&(2,0,1):\textrm{stay}, &(2,0,0):\textrm{stay}.
\end{array}
\end{equation*}
Substituting the above zero variables into (\ref{eq:saf_work5}) and
(\ref{eq:saf_work6}), it can be seen that this policy achieves the
rate pair
\begin{equation*}
r_1= \frac{(1-\epsilon)(3-2\epsilon)}{4(2-\epsilon)},\;\;\;\;\; r_2
=\frac{3-2\epsilon}{4(2-\epsilon)}.
\end{equation*}
\\
\\
$\frac{\alpha_2}{\alpha_1} > (1-\epsilon)(3-2\epsilon)$:
\begin{equation*}
\begin{array}{llll}
\textrm{@queue-2}: (2,1,1):\textrm{stay}, &(2,1,0):\textrm{stay},
&(2,0,1):\textrm{stay}, &(2,0,0):\textrm{stay}.
\end{array}
\end{equation*}
In this case it is optimal to stay at queue-2 for all channel
conditions. Therefore the decisions at queueu-1 are arbitrary.
Namely, it is sufficient that at least one state corresponding to
server being at queue-1 take a switch decision, which is the case
for $\alpha_2/\alpha_1\ge ((1-\epsilon)(3-2\epsilon))$,since $x(3;1)
=0$ if $\alpha_2/\alpha_1\ge 1/((1-\epsilon)(3-2\epsilon))$. Since
the policy always stays at queue-2, it achieves the rate pair
\begin{equation*}
r_1= 0,\;\;\;\;\; r_2 =0.5.
\end{equation*}
Note that the case for $\alpha_2/\alpha_1 < 1$ is symmetric and can
be obtained similarly.\\
\\
\\
\textbf{Case-2: $\epsilon < \epsilon_c = 1-\sqrt{2}/2$}\\
In this case we have $(1-\epsilon)(3-2\epsilon) >
(1+\epsilon-\epsilon^2)/(1-\epsilon)$. This means that before the
state $x(6;0)$ becomes zero, namely for
$(1+\epsilon-\epsilon^2)/(1-\epsilon) < \alpha_2/\alpha_1 <
(1-\epsilon)(3-2\epsilon)$, having $x(1;1)=0$ is optimal. This means
that there is one more corner point of the rate region for $\epsilon
< \epsilon_c$. In more details we have the following optimal
policies.
\\
\\
$1 \le \frac{\alpha_2}{\alpha_1} \le
\frac{1+\epsilon-\epsilon^2}{1-\epsilon}$:\\
\begin{equation*}
\begin{array}{llll}
\textrm{@queue-1}: (1,1,1):\textrm{stay}, &(1,1,0):\textrm{stay}, &(1,0,1):\textrm{switch}, &(1,0,0):\textrm{switch}. \\
\textrm{@queue-2}: (2,1,1):\textrm{stay}, &(2,1,0):\textrm{switch},
&(2,0,1):\textrm{stay}, &(2,0,0):\textrm{stay}.
\end{array}
\end{equation*}
This policy is the same policy as in the previous case and it
achieves the rate pair
\begin{equation*}
r_1= \frac{(1-\epsilon)(3-2\epsilon)}{4(2-\epsilon)},\;\;\;\;\; r_2
=\frac{3-2\epsilon}{4(2-\epsilon)}.
\end{equation*}
\\
\\
$ \frac{\alpha_2}{\alpha_1}
> \frac{1+\epsilon-\epsilon^2}{1-\epsilon}$:\\
\\
We have the following deterministic actions.
\begin{equation*}
\begin{array}{llll}
\textrm{@queue-1}: (1,1,1):\textrm{switch}, &(1,1,0):\textrm{?}, &(1,0,1):\textrm{switch}, &(1,0,0):\textrm{switch}. \\
\textrm{@queue-2}: (2,1,1):\textrm{stay}, &(2,1,0):\textrm{?},
&(2,0,1):\textrm{stay}, &(2,0,0):\textrm{stay}.
\end{array}
\end{equation*}
In order to find the final threshold on $\alpha_2/\alpha_1$, we
substitute the above deterministic decisions in (\ref{eq:saf2}),
(\ref{eq:saf3}) and (\ref{eq:saf4}). Utilizing also
(\ref{eq:state_1}), (\ref{eq:state_2}), (\ref{eq:state_3}) and
(\ref{eq:state_4}) we obtain
\begin{eqnarray}
x(2;1) \!\!\!\!&=&\!\!\!\! \frac{(1-\epsilon)^2}{4} - (1-\epsilon)^2 x(6;1)\label{eq:saf_work7}\\
x(5;1)+x(7;1) \!\!\!\!&=&\!\!\!\! \frac{2-\epsilon}{4} + \epsilon x(6;1)\label{eq:saf_work8}\\
\end{eqnarray}
The previous threshold on $\alpha_2/\alpha_1$ for $x(6;0)$ to be
zero, i.e., $(1-\epsilon)(3-2\epsilon)$, is valid for the case where
$x(1;0)=0$. Other decisions staying the same, when $x(1;0)$ is
positive and $x(1;1)=0$, $r_2$ increases and $r_1$ decreases.
Therefore the threshold on $\alpha_2/\alpha_1$ for $x(6;0)$ to be
zero changes, in particular it becomes $\alpha_2/\alpha_1
> (1-\epsilon)^2/\epsilon$. This gives the following two regions:
\\
\\
$ \frac{1+\epsilon-\epsilon^2}{1-\epsilon} \le
\frac{\alpha_2}{\alpha_1}
\le \frac{(1-\epsilon)^2}{\epsilon}$:\\
\\
The optimal policy is
\begin{equation*}
\begin{array}{llll}
\textrm{@queue-1}: (1,1,1):\textrm{switch}, &(1,1,0):\textrm{stay}, &(1,0,1):\textrm{switch}, &(1,0,0):\textrm{switch}. \\
\textrm{@queue-2}: (2,1,1):\textrm{stay}, &(2,1,0):\textrm{switch},
&(2,0,1):\textrm{stay}, &(2,0,0):\textrm{stay}.
\end{array}
\end{equation*}
From (\ref{eq:saf_work7}) and (\ref{eq:saf_work8}) it is easy to see
that this policy achieves
\begin{equation*}
r_1= \frac{(1-\epsilon)^2}{4},\;\;\;\;\; r_2 =\frac{2-\epsilon}{4}.
\end{equation*}
\\
\\
$\frac{\alpha_2}{\alpha_1}
> \frac{(1-\epsilon)^2}{\epsilon}$:\\
\\
The optimal policy is
\begin{equation*}
\begin{array}{llll}
\textrm{@queue-2}: (2,1,1):\textrm{stay}, &(2,1,0):\textrm{stay},
&(2,0,1):\textrm{stay}, &(2,0,0):\textrm{stay}.
\end{array}
\end{equation*}
This policy achives
\begin{equation*}
r_1= 0,\;\;\;\;\; r_2 =0.5.
\end{equation*}
Similar to Case-1, the case $\alpha_2/\alpha_1 < 1$ is symmetric and
can be solved similarly.

Thus we have characterized the corner point of the stability region for the two regions of $\epsilon$. 
Using these corner points, it is easy to derive the expressions for the lines connecting these corner points, 
which are given in Theorem \ref{thm:stab}.

\section*{Appendix D-Proof of Theorem \ref{thm:FBDC}}\label{Sec:App-D}

\begin{proof}
Let $D_i(t)$ be $1$ if there is a departure from queue-$i$ at time
slot $t$ and zero otherwise, we have the following queue evolution
relation.
\begin{equation*}
Q_i(t+1) = Q_i(t) + A_i(t) - D_i(t).
\end{equation*}
Writing similar expressions for time slots $t \in \{t+2,...,t+T \}$
and summing all the expressions creates a telescoping series,
yielding
\begin{equation*}
\displaystyle Q_i(t+T) = Q_i(t)  - \sum_{\tau=0}^{T-1}D_i(t+\tau) +
\sum_{\tau=0}^{T-1}A_i(t+\tau).
\end{equation*}
Taking the square of both sides we obtain
\begin{equation}\label{eq:drift1}
\displaystyle Q_i(t+T)^2 \le Q_i(t)^2  +
\Big(\sum_{\tau=0}^{T-1}D_i(t+\tau) \Big)^2 +
\Big(\sum_{\tau=0}^{T-1}A_i(t+\tau) \Big)^2 - 2Q_i(t)
\Big(\sum_{\tau=0}^{T-1}D_i(t+\tau)
-\sum_{\tau=0}^{T-1}A_i(t+\tau)\Big).
\end{equation}
Define the quadratic Lyapunov function \beq L(\mathbf{Q}(t)) =
\sum_{i=1}^2 Q_i^2(t),\nonumber \eeq and the $T$-step conditional
Lyapunov drift
\begin{equation*}
\Delta_T(t)\triangleq \mathbb{E} \left\{ L(\mathbf{Q}(t+T)) -
L(\mathbf{Q}(t))\big|\mathbf{Q}(t)\right\}.
\end{equation*}
Summing (\ref{eq:drift1}) over both queues, taking conditional
expectation, using $D_i(t)\le 1$ for all time slots $t$,
$\mathbb{E}\{A_i(t)^2\} \le A_{\max}^2$ and
$\mathbb{E}\{A_i(t_1)A_i(t_2)\} \le
\sqrt{\mathbb{E}\{A_i(t_1)\}^2\mathbb{E}\{A_i(t_2)\}^2}\le
A_{\max}^2$ for all $t_1$ and $t_2$ we have \beqn \Delta_T(t)
\!\!\!\!\!\!\!\!&\le\!\!\!\!\!\!\!& 2BT^2 \!\!+\!\! 2\mathbb{E}
\left\{ \sum_{i} Q_i(t)\!\! \sum_{\tau=0}^{T-1} \left[A_i(t+\tau)
- D_i(t+\tau) \right] \big|\mathbf{Q}(t)\!\!\right\}\nonumber \\
& = &   2BT^2 + 2T \sum_{i} Q_i(t) \lambda_i \!\!-\!\! 2\sum_{i}
Q_i(t) \mathbb{E}\left\{ \sum_{\tau=0}^{T-1} D_i(t+\tau)
\big|\mathbf{Q}(t)\!\! \right\} \nonumber\eeqn where
$B=1+A_{\max}^2$ is a constant. 

Let $r_i(t)$ be a reward function such that $r_i(t)=1$ if at time
$t$ the server is at queue-$i$ with ON channel and decides to stay
at queue-$i$ at time $t$ and $r_i(t)=0$ otherwise. Note that
$r_i(t)$ is simply the reward function associated with applying
policy $\pi^*$ to the saturated queue system whose infinite horizon
average rate is $\mathbf{r}^*=(r_1^*,r_2^*)$. Let $\mathbf{x}^*$ be
the optimal vector of state action frequencies corresponding to
$\pi^*$. Define the time average empirical reward from queue-$i$ in
the saturated system, $\hat{r}_{T,i}(t)$, and that
in the actual system, $\hat{D}_{T,i}(t)$, as\\
\beq \hat{r}_{ T,i }(t) \triangleq \frac{1}{T}\sum_{\tau=0}^{T-1}
r_i(t+\tau),\;\;\;\hat{D}_{ T,i }(t) \triangleq
\frac{1}{T}\sum_{\tau=0}^{T-1} D_i(t+\tau)\nonumber. \eeq
Also define the corresponding two dimensional vectors
$\hat{\mathbf{r}}_{T}(t)$ and $\hat{\mathbf{D}}_{ T }(t)$. Similarly
define time average empirical state action frequency vector
$\hat{\mathbf{x}}_{ T}(t)$. Let $\mathbf{a.b}$ denote the inner
product for vectors $\mathbf{a}$ and $\mathbf{b}$. From the
definition of the rewards in terms of state action frequencies in
(\ref{eq:dept_rate}) we can write $\hat{r}_{ T,1 }(t) =
\mathbf{a_1}.\hat{\mathbf{x}}_{ T}(t)$, $\hat{r}_{ T,2 }(t) =
\mathbf{a_2}.\mathbf{\hat{x}}_{ T}(t)$ and $r_1^*=\mathbf{a_1.x}^*$,
$r_2^*=\mathbf{a_2.x}^*$, where $\mathbf{a_1}$ and $\mathbf{a_2}$
are appropriate vectors of dimension 16. Now we have that as $T$
increases, $\hat{\mathbf{x}}_{ T}(t)$ converges to $\mathbf{x}^*$
and hence $\mathbf{\hat{r}}_T(t)$ converges to $\mathbf{r}^*$ with
probability 1 regardless of the initial state of the system. More
precisely we have the following lemma \cite{GlynnOrm02},
\cite{shie05}:
\begin{lemma}\label{lem:MDP_convergence}
\emph{For every choice of initial state distribution, there exists
constants $c_1$ and $c_2$ such that}
\begin{equation*}
\mathbf{P}(||\hat{\mathbf{x}}_{ T}(t)-\mathbf{x}^*|| \ge \delta_1)
\le c_1 e^{-c_2\delta_1^2 T}, \;\;\; \forall T\ge1,\; \forall
\delta_1
>0.
\end{equation*}
\end{lemma}
Therefore using
$||\hat{\mathbf{r}}_{ T}(t)-\mathbf{r}^*|| \le ||\hat{\mathbf{x}}_{
T}(t)-\mathbf{x}^*||\big(||a_1||^2 + ||a_2||^2\big)^{\frac{1}{2}},$
we have that there exists constants $c_2$ and $c_3$ such that
\begin{equation}\label{eq:MDP_rate_conv}
\mathbf{P}(||\hat{\mathbf{r}}_{ T}(t)-\mathbf{r}^*|| \ge \delta_1)
\le c_1 e^{-c_3\delta_1^2 T}, \; \forall T\ge1, \forall \delta_1
>0,
\end{equation}
under policy $\pi^*$ for any initial state distribution.
%
%
Now define the following: \beqn W_T(t)  \!\!\!&\triangleq&\!\!\!
\sum_{i} Q_i(t) \frac{1}{T}\sum_{\tau=0}^{T-1} D_i(t+\tau) =
\sum_{i} Q_i(t)
\hat{D}_{ T,i }(t).\nonumber\\
R_T(t)  \!\!\!&\triangleq&\!\!\! \sum_{i} Q_i(t)
\frac{1}{T}\sum_{\tau=0}^{T-1}
r_i(t+\tau) = \sum_{i} Q_i(t) \hat{r}_{ T,i }(t). \nonumber\\
R^*(t) \!\!\!&\triangleq&\!\!\! \sum_{i} Q_i(t) r_i^*.\nonumber
\eeqn We rewrite the drift expression as
\begin{eqnarray*}
\frac{\Delta_T(t)}{2T} \leq BT \!\!\!\!\!\!\!\!\!&&+ \sum_{i} Q_i(t)
\lambda_i -
\mathbb{ E}\left\{ W_T(t) \big|\mathbf{Q}(t) \right\}\\
= BT \!\!\!\!\!\!\!\!\!&&+ \sum_{i} Q_i(t) \lambda_i -  \mathbb{ E}
\left\{
R^*(t)\big|\mathbf{Q}(t) \right\}\nonumber\\
\!\!\!\!\!\!\!\!\!&&+ \mathbb{ E} \left\{ R^*(t)-W_T(t)\big|\mathbf{Q}(t) \right\}\\
= BT \!\!\!\!\!\!\!\!\!&&+ \sum_{i} Q_i(t) \lambda_i -  \sum_{i} Q_i(t) r_i^*\\
\!\!\!\!\!\!\!\!\!&&+ E \left\{ R^*(t)-W_T(t)\big|\mathbf{Q}(t)
\right\}.
\end{eqnarray*}
Now we bound the last term. \beqn
&\!\!\!\!\!\!\!\!\!\!\!\!\!\!\!\!\!\!\!\!\!\!\!\mathbb{ E}&
\!\!\!\!\!\!\!\!\!\!\!\!\!\!\!\!\!\!\!\!\!\!\!\left\{
R^*(t)-W_T(t)\big|\mathbf{Q}(t) \right\} = \nonumber\\
&\!\!\!\!\!\!\!\!\!=&\!\!\!\!\!\!\!\!\!\!\!\!\!\mathbb{ E}\left\{ R^*(t)-W_T(t)\big|\mathbf{Q}(t), R^*(t)-W_T(t) \ge \delta_2||\mathbf{Q(t)}|| \right\}\nonumber\\
&\;\;\;\;\;\;\;\;\;.& \!\!\!\!\mathbf{P}\left(R^*(t)-W_T(t) \ge \delta_2||\mathbf{Q(t)}||\; \big|\mathbf{Q}(t)\right) \nonumber\\
&\;\;\;\;\;\;\;+& \!\!\!\! \mathbb{E} \left\{ R^*(t)-W_T(t\big)\big|\vec{Q}(t), R^*(t)-W_T(t) < \delta_2||\mathbf{Q(t)}|| \right\}\nonumber\\
&\;\;\;\;\;\;\;\;\;.& \!\!\!\!\mathbf{P}\left(R^*(t)-W_T(t) < \delta_2||\mathbf{Q(t)}|| \; \big| \mathbf{Q}(t)\right) \nonumber\\
&\!\!\!\!\!\!\!\!\le&\!\!\!\!\!\!\!\!\!\!\!\!\!\!\!\!\!
\Big(\!\!\sum_{i}Q(t) \!\!\Big) \mathbf{P}\!\left
(|R^*(t)\!\!-\!\!W_T(t)| \ge \delta_2||\mathbf{Q(t)}|| \; \big|
\mathbf{Q}(t)\right)\!\!+\delta_2 ||\mathbf{Q}(t)||. \ \nonumber
\eeqn Consider
\begin{eqnarray*}
&\!\!\!\!\!\!\!\!\!\!\!\!\!\!\!\!\!\!\!\!\!\!\!\!\mathbf{P}&\!\!\!\!\!\!\!\!\!\!\!\!\!\!\!\!\!\!\!\!\!\!\!
\left(|R^*(t)-W_T(t)|\ge\delta_2||\mathbf{Q(t)}||\;\big| \mathbf{Q}(t)\right)  \\
&\!\!\!\!\!\!\!\!\le& \!\!\!\!\!\!\!\!\!\!\!\!\!\!\mathbf{P}\left(|R^*(t)-R_T(t)| \ge \frac{\delta_2}{2}||\mathbf{Q(t)}|| \; \big| \mathbf{Q}(t)\right) \\
&\;\;\;\;\;\;+&\!\!\!\! \mathbf{P}\left(|W_T(t)-R_T(t)| \ge \frac{\delta_2}{2}||\mathbf{Q(t)}|| \; \big| \mathbf{Q}(t)\right)\\
&\!\!\!\!\!\!\!\!\le& \!\!\!\!\!\!\!\!\!\!\!\!\!\!
\mathbf{P}\left(\!\!||\mathbf{r}^*-\mathbf{\hat{r}}_T(t)|| \!\ge
\frac{\delta_2}{2} \big| \mathbf{Q}(t)\right)\\
&\;\;\;\;\;\;+&\!\!\!\!
\mathbf{P}\left(\!||\mathbf{\hat{D}}_T(t)\!-\!\mathbf{\hat{r}}_T(t)||
\!\ge\! \frac{\delta_2}{2} \big| \!\mathbf{Q}(t)\!\!\right)
\end{eqnarray*} where the last inequality follows from the Schwartz
Inequality for inner products given as \beqn
|R^*(t)-R_T(t)|\!=\!|\mathbf{Q}(t).(\mathbf{r}^*\!-\!\mathbf{\hat{r}}_T(t))|\!\le\!||\mathbf{Q}(t)||.||\mathbf{r}^*\!-\!\mathbf{\hat{r}}_T(t)||\nonumber.
\eeqn
 Using (\ref{eq:MDP_rate_conv}), there exists constant $c_4$ such
 that
\beqn \mathbb{E} \left\{ R^*(t)\!-\!W_T(t)\big|\mathbf{Q}(t)
\right\} &\!\!\!\!\!\le&\!\! \!\!\!\!\Big(\sum_{i}Q(t)
\Big)c_1e^{-c_4\delta_2 ^2T}\!\! +\!\delta_2 ||\mathbf{Q}(t)||
\nonumber\\
&\!\!\!\!\!\!\!\!\!\!\!\!\!\!\!\!\!\!\!\!\!\!\!\!\!\!\!\!\!\!\!\!\!\!\!\!\!\!\!\!\!\!\!\!\!\!\!\!\!\!\!\!\!\!\!\!\!\!\!+&
\!\!\!\!\!\!\!\!\!\!\!\!\!\!\!\!\!\!\!\!\!\!\!\!\!\!\!\!\!\!\!\Big(\!\!\sum_{i}Q(t)\!\!
\Big)
\mathbf{P}\!\left(||\mathbf{\hat{D}}_T(t)\!-\!\mathbf{\hat{r}}_T(t)||
\ge \frac{\delta_2}{2} \big| \mathbf{Q}(t)\!\!\right)\!\!. \
\nonumber \eeqn Hence we can write the drift term as
\begin{eqnarray*}
\frac{\Delta_T(t)}{2T} \leq BT\!\!\!\!\!\!\!\!&\;\,+& \!\!\!\!
\sum_{i} Q_i(t) \lambda_i - \sum_{i} Q_i(t) r_i^* +
\Big(\sum_{i}Q(t) \Big)\\
&\!\!\!\!\!\!\!\!\!\!\!\!\!\!\!\!\!\!\!\!\!\!.&\!\!\!\!\!\!\!\!\!\!\!\!\!\!\!\!\!\!\!\!\left(c_1e^{-c_4\delta_2
^2T} \!\!+ \delta_2\!+\!
\mathbf{P}\!\left(\!||\hat{D}_T(t)-\hat{r}_T(t)|| \ge
\!\frac{\delta_2}{2}\;
\big|\mathbf{Q}(t)\!\right)\!\!\!\right)\!\!.\nonumber
\end{eqnarray*}
Note that $||\hat{D}_T(t)-\hat{r}_T(t)||$ is because of the lost
rewards due to empty queues and it is equal to zero if both of the
queues have more than $T$ packets at time $t$. Namely,
$||\hat{D}_T(t)-\hat{r}_T(t)||=0$ if $Q_1(t)\ge T$ and $Q_2(t)\ge
T$. Therefore, calling $\delta \triangleq c_1e^{-c_4\delta_2 ^2T} +
\delta_2$, we can write \beqn \frac{\Delta_T(t)}{2T}
\!\!\!\!\!&\leq&\!\! \!\!\!BT \!+\! \!\sum_{i} \!Q_i(t) \lambda_i\!
-\!\!\! \sum_{i} \!Q_i(t) r_i^* \!+\!\!
(\delta \!\!+\!\! 1_{\{\mathbf{Q}(t) < T.\mathbf{1}\}}\!)\!\!\!\sum_{i}Q(t)\! \nonumber\\
&\!\!\!\leq& \!\!\!\!\!\!BT \!\!+\!\! \sum_{i} Q_i(t) \lambda_i\!
-\! \sum_{i} Q_i(t) r_i^* \!+\! \delta\sum_{i}Q_i(t)  +
2T.\label{eq:drift_comparison} \eeqn Now for $(\lambda_1,\lambda_2)$
strictly inside the $\mathbf{\delta}$-\emph{stripped} throughput
region $\mathbf{\Lambda_s^{\delta}}$,
there exist a small $\xi >0$ such that $(\lambda_1,\lambda_2) +
(\xi,\xi) = (r_1,r_2) -(\delta,\delta)$, for some
$\mathbf{r}=(r_1,r_2) \in \mathbf{\Lambda_s}$. Therefore we have,
\begin{eqnarray*}
\frac{\Delta_T(t)}{2T} \leq (B+2)T + \sum_{i} Q_i(t)(r_i- r_i^*)-
 \xi\sum_{i}Q_i(t).
\end{eqnarray*}
Finally using $\sum_{i} Q_i(t)(r- r_i^*) \le 0$ we have \beqn
\frac{\Delta_T(t)}{2T} \leq (B+2)T - \Big(\sum_{i}Q_i(t) \Big)\xi.
\nonumber \eeqn Hence the queue sizes
have negative drift when $\sum_{i} Q_i(t)$ is outside a bounded set.
Therefore the
system is stable for $\mathbf{\lambda }$ within the
$\mathbf{\delta}$-\emph{stripped} stability region
$\mathbf{\Lambda_s^{\delta}}$ where $\delta(T)$ is a decreasing
function of $T$ (see e.g., \cite[Theorem 3]{neely05}).
Note that $\delta = c_1e^{-c_4 \delta_2^2T} + \delta_2$ for any $\delta_2 >0$. 
Therefore choosing $\delta_2$ appropriately (for example, $\delta_2=T^{-0.5+\delta_3}$ for some small $\delta_3>0$), we have that $\delta(T)$ is a decreasing function of $T$.
\end{proof}

\section*{Appendix E-Proof of Lemma \ref{lem:My_FBDC_comp}}\label{Sec:App-E}

Here we prove that $\Psi' \ge 0.9002$ where
\begin{equation*}
\Psi'=\frac{\sum_{i} Q_i(t) r_i^{My}}{\sum_{i} Q_i(t) r_i^{*}}.
\end{equation*}
\begin{proof}
We divide the proof into separate cases for different $\epsilon$
regions.
\subsubsection{Weighted Departure-Rate Ratio Analysis, Case 1: $\epsilon <\epsilon_c$}
Considering the mappings in figures \ref{Fig:myopic_map} and
\ref{Fig:optimal_map}, the regions where the Myopic policy and the
optimal policy ``chooses'' the same corner point, we have $\Psi'=1$.
In the following we analyze the ratio in the regions where the two
policies chooses different corner points. We term these cases as
``discrepant'' cases. We will use $Q_1$ and $Q_2$ instead of
$Q_1(t)$ and $Q_2(t)$ for notational simplicity. 
Note that we have that $\frac{2-\epsilon}{1-\epsilon}
> \frac{(1+\epsilon-\epsilon^2)}{(1-\epsilon)}$ always holds.
However $\frac{2-\epsilon}{1-\epsilon}$ equals
$\frac{(1-\epsilon)^2}{\epsilon}$ at $\epsilon_t = 0.245$ for the
case of $\epsilon < \epsilon_c = 0.293$.\\
\\
\textbf{Case 1.1: $\epsilon < \epsilon_t \rightarrow
\;$$\frac{2-\epsilon}{1-\epsilon} <
\frac{(1-\epsilon)^2}{\epsilon}$}\\
\textit{Discrepant Region 1: $\frac{(1-\epsilon)^2}{\epsilon} < \frac{Q_2}{Q_1}< \frac{1-\epsilon}{\epsilon}$}\\
In this case the Myopic policy chooses the corner point $b_1$
whereas the optimal policy chooses the corner point $b_0$.
Therefore,
\begin{eqnarray*}
\Psi' \!\!\!\!\!&=&\!\!\!\! \frac{Q_1\big(
\frac{(1-\epsilon)^2}{4}\big)+ Q_2 \big(
\frac{1}{2}-\frac{\epsilon}{4}\big)}  {Q_2\frac{1}{2}} \ge 1-
\frac{\epsilon}{2}+
\frac{(1-\epsilon)^2}{2}\frac{\epsilon}{1-\epsilon}\nonumber\\
&=& 1-\frac{\epsilon^2}{2} \ge 0.9700.
\end{eqnarray*}
\textit{Discrepant Region 2: $\frac{(1+\epsilon-\epsilon^2)}{1-\epsilon} < \frac{Q_2}{Q_1}< \frac{2-\epsilon}{1-\epsilon}$}\\
In this case the Myopic policy chooses the corner point $b_2$
whereas the optimal policy chooses the corner point $b_1$.
Therefore,
\begin{eqnarray*}
\Psi' \!\!\!\!\!&=&\!\!\!\! \frac{Q_1\big(
\frac{3}{8}-\frac{\epsilon}{2}+\frac{\epsilon}{8(2-\epsilon)}\big)+
Q_2 \big( \frac{3}{8}-\frac{\epsilon}{8(2-\epsilon)}\big)} {Q_1\big(
\frac{(1-\epsilon)^2}{4}\big)+ Q_2 \big(
\frac{1}{2}-\frac{\epsilon}{4}\big)} \\
&=& \frac{
\frac{3}{8}-\frac{\epsilon}{2}+\frac{\epsilon}{8(2-\epsilon)}+
\frac{Q_2}{Q_1} \big(
\frac{3}{8}-\frac{\epsilon}{8(2-\epsilon)}\big)} {
\frac{(1-\epsilon)^2}{4}+ \frac{Q_2}{Q_1} \big(
\frac{1}{2}-\frac{\epsilon}{4}\big)} \ge 0.9002.
\end{eqnarray*}
This is a minimization of a function of two variables for all
possible $\epsilon$ values in the interval $0 \le \epsilon \le
\epsilon_t$, and the ratio $\frac{Q_2}{Q_1}$ in the interval
$\frac{(1+\epsilon-\epsilon^2)}{1-\epsilon} < \frac{Q_2}{Q_1}<
\frac{2-\epsilon}{1-\epsilon}$.
\\
\\
\textbf{CASE 1.2: $\epsilon_t < \epsilon < \epsilon_c  \rightarrow
\;$$\frac{2-\epsilon}{1-\epsilon} >
\frac{(1-\epsilon)^2}{\epsilon}$}\\
\textit{Discrepant Region 1: $\frac{(2-\epsilon)}{(1-\epsilon)} < \frac{Q_2}{Q_1}< \frac{1-\epsilon}{\epsilon}$}\\
In this case the Myopic policy chooses the corner point $b_1$
whereas the optimal policy chooses the corner point $b_0$.
Therefore,
\begin{eqnarray*}
\Psi'\!\!\!\!\!&=&\!\!\!\! \frac{Q_1\big(
\frac{(1-\epsilon)^2}{4}\big)+ Q_2 \big(
\frac{1}{2}-\frac{\epsilon}{4}\big)}  {Q_2\frac{1}{2}} \ge 1-
\frac{\epsilon}{2}+
\frac{(1-\epsilon)^2}{2}\frac{\epsilon}{1-\epsilon}\\
&=& 1-\frac{\epsilon^2}{2} \ge 0.9500.
\end{eqnarray*}
\textit{Discrepant Region 2: $\frac{(1-\epsilon)^2}{\epsilon} < \frac{Q_2}{Q_1}< \frac{2-\epsilon}{1-\epsilon}$}\\
In this case the Myopic policy chooses the corner point $b_2$
whereas the optimal policy chooses the corner point $b_0$.
Therefore,
\begin{eqnarray*}
\Psi'\!\!\!\!\!&=&\!\!\!\! \frac{Q_1\big(
\frac{3}{8}-\frac{\epsilon}{2}+\frac{\epsilon}{8(2-\epsilon)}\big)+
Q_2 \big( \frac{3}{8}-\frac{\epsilon}{8(2-\epsilon)}\big)} { Q_2
\frac{1}{2}} \\
&\ge&
\!\!\!\!\!\big(\frac{1-\epsilon}{2-\epsilon}\big)\big(\frac{3}{4}-\epsilon+\frac{\epsilon}{4(2-\epsilon)}\big)
+ \frac{3}{4}-\frac{\epsilon}{4(2-\epsilon)} \ge 0.9150.
\end{eqnarray*}
\textit{Discrepant Region 3: $\frac{(1+\epsilon-\epsilon^2)}{1-\epsilon} < \frac{Q_2}{Q_1}< \frac{(1-\epsilon)^2}{\epsilon}$}\\
In this case the Myopic policy chooses the corner point $b_2$
whereas the optimal policy chooses the corner point $b_1$.
Therefore,
\begin{eqnarray*}
\Psi'\!\!\!\!\!&=&\!\!\!\! \frac{Q_1\big(
\frac{3}{8}-\frac{\epsilon}{2}+\frac{\epsilon}{8(2-\epsilon)}\big)+
Q_2 \big( \frac{3}{8}-\frac{\epsilon}{8(2-\epsilon)}\big)} {Q_1\big(
\frac{(1-\epsilon)^2}{4}\big)+ Q_2 \big(
\frac{1}{2}-\frac{\epsilon}{4}\big)} \\
&\ge& \frac{
\frac{3}{8}-\frac{\epsilon}{2}+\frac{\epsilon}{8(2-\epsilon)}+
\frac{Q_2}{Q_1} \big(
\frac{3}{8}-\frac{\epsilon}{8(2-\epsilon)}\big)} {
\frac{(1-\epsilon)^2}{4}+ \frac{Q_2}{Q_1} \big(
\frac{1}{2}-\frac{\epsilon}{4}\big)} \ge 0.9474.
\end{eqnarray*}

\subsection{Weighted Departure-Rate Ratio Analysis, Case 2: $\epsilon_c <
\epsilon < 0.5$} Considering the mappings in figures
\ref{Fig:myopic_map2} and \ref{Fig:optimal_map2}, again for the
regions where the Myopic policy and the optimal policy ``chooses''
the same corner point, we have $\Psi'=1$. We analyze the ratio in
the regions where the two policies chooses different corner points
termed as ``discrepant'' cases. Note that
$(1-\epsilon)(3-2\epsilon)$ is always less than or equal to
$(1-\epsilon)/\epsilon$ for $\epsilon \ge \epsilon_c$. Since due to
$\epsilon > \epsilon_c$ we also have $\frac{1-\epsilon}{\epsilon} <
\frac{2-\epsilon}{1-\epsilon}$, there is only one discrepancy
region.\\
\textit{Discrepant Region 1: $(1-\epsilon)(3-2\epsilon) < \frac{Q_2}{Q_1}< \frac{1-\epsilon}{\epsilon}$}\\
In this case the Myopic policy chooses the corner point $b_1$
whereas the optimal policy chooses the corner point $b_0$.
Therefore,
\begin{eqnarray*}\label{eq:Psi_6}
\Psi'\!\!\!\!\!&=&\!\!\!\! \frac{Q_1\big(
\frac{3}{8}-\frac{\epsilon}{2}+\frac{\epsilon}{8(2-\epsilon)}\big)+
Q_2 \big( \frac{3}{8}-\frac{\epsilon}{8(2-\epsilon)}\big)}
{Q_2\frac{1}{2}} \\
&\ge&
\!\!\!\!\big(\frac{\epsilon}{1-\epsilon}\big)\big(\frac{3}{4}-\epsilon+\frac{\epsilon}{4(2-\epsilon)}\big)
+ \frac{3}{4}-\frac{\epsilon}{4(2-\epsilon)} \ge 0.914.
\end{eqnarray*}\\
Combining all the cases, for all $\epsilon \in [0,0.5]$, we have
that $\Psi' \ge 0.9002$ for all possible $Q_1$ and $Q_2$.
\end{proof}

\bibliographystyle{IEEE}

\begin{thebibliography}{1}

\bibitem{AhmadTaraKrish09} S. Ahmad, L. Mingyan, T. Javidi, Q. Zhao, and B. Krishnamachari,
``Optimality of Myopic Sensing in Multichannel Opportunistic
Access,'' \emph{IEEE Trans. Infor. Theory}, vol.~55, no. 9,
pp.~4040-4050, Sept. 2009.
%
%


\bibitem{AltKonsLiu92} E. Altman, P. Konstantopoulos, and Z. Liu, ``Stability, monotonicity and invariant quantities in
general polling systems,'' \emph{Queuing Sys.}, vol.~11, pp.~35-57,
Mar. 1992.

\bibitem{AltKush00} E. Altman and H. J. Kushner, ``Control of polling in presense of vacations in heavy traffic with
applications to satellite and mobile radio systems,'' \emph{SIAM J.
on Control and Opt.}, vol.~41, pp.~217-252, 2002.

\bibitem{BlakeLong09} L. Blake and M. Long, ''Antennas: Fundamentals, Design, Measurement,'' \emph{SciTech}, 2009.

\bibitem{Berz} A.~Brzezinski, ``Scheduling algorithms for throughput maximization in data networks,'' Ph.D. thesis, MIT, 2007.


\bibitem{ChapKarSar05} P. Chaporkar, K. Kar, and S. Sarkar, ``Throughput guarantees through maximal scheduling
in wireless networks,'' In \emph{Proc. Allerton'05}, Sept. 2005.

\bibitem{Long10}
L. B. Le, E. Modiano, C. Joo, and N. B. Shroff,
``Longest-queue-first scheduling under SINR interference model,'' In
\emph{Proc. ACM MobiHoc'10}, Sept. 2010, to appear.
%

\bibitem{EryilOzMod07} A. Eryilmaz, A. Ozdaglar, and E. Modiano, ``Polynomial complexity algorithms for full utilization of multi-hop
wireless networks,'' In \emph{Proc. IEEE Infocom'07}, May. 2007.
%
\bibitem{Geor_Neely_Tass06} L. Georgiadis, M. Neely, and L. Tassiulas, ``Resource Allocation and Cross-Layer Control in Wireless Networks,'' Now Publishers, 2006.

\bibitem{GlynnOrm02} P. W. Glynn and D. Ormoneit, ``Hoeffding's inequality for uniformly ergodic Markov chains,''
\emph{Stat. and Poly. Letters}, vol. 56, pp. 143-146, 2002.

\bibitem{KarLuo07} K. Kar, X. Luo, and S. Sarkar, ``Throughput-optimal scheduling in
multichannel access point networks under infrequent channel
measurements,'' In \emph{Proc. IEEE Infocom'07}, May. 2007.

\bibitem{Levy} H. Levy, M. Sidi, and O.L. Boxma, ``Dominance relations in polling systems,''
\emph{Queueing Systems}, vol. 6, pp. 155-172, Apr. 1990.

\bibitem{LiNeely10} C. Li and M. Neely, ``On achievable network capacity and throughput-achieving policies over Markov
ON/OFF channels,'' In \emph{Proc. WiOpt'10}, Jun. 2010.

\bibitem{LinShr05} X. Lin and N. B. Shroff, ``The impact of imperfect scheduling on
cross-layer rate control in wireless networks,'' In \emph{Proc. IEEE
Infocom'05}, Mar. 2005.


\bibitem{LiuNainTow} Z. Lui, P. Nain, and D. Towsley, ``On optimal polling policies,'' \emph{Queuing Sys.}, vol. 11, pp.~59-83, Jul. 1992.

\bibitem{shie05}
S.~Mannor and J.~N.~Tsitsiklis, ``On the emperical state-action
frequencies in Markov Decision Processes under general policies,''
{\em Mathematics of Operation Research}, vol. 30, no. 3, Aug. 2005.

\bibitem{ModBarry00}
E. Modiano and R. Barry, ``A novel medium access control protocol
for WDM-based LAN's and access networks using a Master/Slave
scheduler,'' {\em IEEE J. Lightwave Tech.}, vol. 18, no. 4, pp.
461--468, Apr. 2000.

\bibitem{ModShahZuss06} E. Modiano, D. Shah and G. Zussman, ``Maximizing throughput in
wireless networks via Gossip,'' In \emph{Proc. ACM
SIGMETRICS/Performance'06}, June 2006.


\bibitem{NavDas07} V. Navda, A. Subramanian, K. Dhanasekaran, A. Timm-Giel, and S. Das,
``MobiSteer: Using Steerable Beam Directional Antenna for
Vehicular Network Access,'' In \emph{Proc. ACM MobiSys}, Jun. 2007.

\bibitem{neely03}
M. J. Neely, E. Modiano, and C. E. Rohrs, ``Power allocation and
routing in multi-beam satellites with time varying channels,'' {\em
IEEE Trans. Netw.}, vol. 11, no. 1, pp. 138--152, Feb. 2003.


\bibitem{neely05}
M. J. Neely, E. Modiano, and C. E. Rohrs, ``Dynamic power allocation
and routing for time varying wireless networks,'' {\em IEEE J. Sel.
Areas Commun.}, vol. 23, no. 1, pp. 89--103, Jan. 2005.

\bibitem{neely05inf}
M.~Neely, E.~Modiano, and C.~Li, ``Fairness and optimal stochastic
control for heterogeneous networks,'' In \emph{Proc. IEEE
Infocom'05}, Mar. 2005.

\bibitem{puterman05} M. Puterman, ''Markov decision processes,'' \emph{Wiley}, 2005.


\bibitem{ShahWis06} D. Shah and D. J. Wischik, ``Optimal scheduling algorithms for input-queued switches,'' In \emph{Proc. IEEE
Infocom'06}, Mar. 2006.



\bibitem{AnnaEphr09}
A. Pantelidou, A. Ephremides, and A. Tits ``A cross-layer approach
for stable throughput maximization under channel state
uncertainty,'' {\em Wireless Networks}, vol. 15, no.5, pp. 555-569,
Jul. 2009.

\bibitem{Stolyar04}
A. L. Stolyar, ``Maxweight scheduling in a generalized switch: State
space collapse and workload minimization in heavy traffic,'' {\em
Annals of Appl. Prob.}, vol. 14, no. 1, pp. 1–-53, 2004.
%
\bibitem{Takagi} H. Takagi, ``Queueing analysis of polling models,''
\emph{ACM Computing Surveys}, pp.~5-28, no. 1, Mar. 1988.

\bibitem{tass92}
L.~Tassiulas and A.~Ephremides, ``Stability properties of
constrained queueing systems and scheduling policies for maximum
throughput in multihop radio networks,'' {\em IEEE Trans. Auto.
Control}, vol. 37, no. 12, pp. 1936-1948, Dec. 1992.

\bibitem{tass93}
L.~Tassiulas and A.~Ephremides, ``Dynamic server allocation to
parallel queues with randomly varying connectivity,'' {\em IEEE
Trans. Infor. Theory}, vol. 39, no. 2, pp. 466-478, Mar. 1993.

\bibitem{TolShu96} A. Tolkachev, V. Denisenko, A. Shishlov, and A. Shubov, ``High gain antenna systems for millimeter wave radars
with combined electronical and mechanical beam steering,'' In
\emph{Proc. IEEE Symp. Phased Array Sys. Tech.}, Oct. 2006.

\bibitem{vish06}
V. M. Vishnevskii and O. V. Semenova, ``Mathematical methods to
study the polling systems,'' {\em Auto. and Rem. Cont.}, vol. 67,
no. 2, pp. 173-220, Feb. 2006.
%

\bibitem{Walrand} J. Walrand, ``Queuing Networks,'' Englewood Cliffs, NJ:Prentice Hall, 1988.


\bibitem{WangChang96}
H.~Wang and P.~Chang, ``On verifying the first-order Markovian
assumption for a Rayleigh fading channel model,'' {\em IEEE Trans.
Veh. Tech.}, vol. 45, no. 2, pp. 353-357, May 1996.


\bibitem{WuSri06} X. Wu and R. Srikant, ``Bounds on the capacity
region of multi-hop wireless networks under distributed greedy
scheduling,'' In \emph{Proc. IEEE Infocom'06}, Mar. 2006.

\bibitem{shakk08}
L. Ying, and S. Shakkottai, ``On throughput optimality with delayed
network-state information,'' In {\em Proc. Inform. Theory and
Applic. Workshop}, Jan. 2008.

\bibitem{ZorziRao95}
M. Zorzi, R. Rao, and L. Milstein, ``On the accuracy of a
first-order Markov model for data transmission on fading channels,''
In {\em Proc. ICUPC'95}, 1995.

\bibitem{ZorziRao97}
M. Zorzi, R. Rao, and L. Milstein, ``ARQ error control for fading
mobile radio channels,'' {\em IEEE Trans. Veh. Tech.}, vol. 46, pp.
445–-455, May 1997.

\end{thebibliography}

\end{document}